\documentclass[11pt,a4paper,reqno,dvipsnames]{amsart}
\usepackage{geometry}
\usepackage{amsfonts}
\usepackage{amsmath}
\usepackage{amssymb,stmaryrd}
\usepackage{amsthm}
\usepackage{amsxtra}
\usepackage{bbm}
\usepackage{braket}
\usepackage{comment}
\usepackage{extarrows}
\usepackage{graphicx}
\usepackage{latexsym}
\usepackage{mathrsfs}
\usepackage{yhmath}

\usepackage{mathtools}

\usepackage{float}
\usepackage{xcolor}

\usepackage[normalem]{ulem}

\usepackage{tikz}
\usetikzlibrary{patterns}

\usepackage[pdfborder={0 0 0}]{hyperref} 
\hypersetup{colorlinks, citecolor=blue, linkcolor=blue, urlcolor=ForestGreen}
\hypersetup{pagebackref}
\usepackage{bm}

\usepackage[initials,lite]{amsrefs} 
\AtBeginDocument{%
    \def\MR#1{}
}
\usepackage{cleveref}
\Crefname{Lemma}{Lemma}{Lemmas}
\Crefname{Theorem}{Theorem}{Theorems}

\makeatother

\theoremstyle{plain}
\newtheorem{Question}{Question}

\newtheorem{Theorem}{Theorem}[section]
\newtheorem*{Theorem*}{Theorem}
\newtheorem{Lemma}[Theorem]{Lemma}
\newtheorem{Corollary}[Theorem]{Corollary}
\newtheorem{Proposition}[Theorem]{Proposition}

\theoremstyle{definition}

\newtheorem{Assumptions and Discussion}[Theorem]{Assumptions and Discussion}

\newtheorem{Example}[Theorem]{Example}
\newtheorem{Definition}[Theorem]{Definition}

\newtheorem{Remark}[Theorem]{Remark}

\newtheorem{Notation}[Theorem]{Notation}

\theoremstyle{remark}

\newtheorem*{acknowledgment*}{Acknowledgment}

\usepackage[shortlabels]{enumitem}
\SetEnumerateShortLabel{a}{\textup{(\alph*)}}
\SetEnumerateShortLabel{A}{\textup{(\Alph*)}}
\SetEnumerateShortLabel{1}{\textup{(\arabic*)}}
\SetEnumerateShortLabel{i}{\textup{(\roman*)}}
\SetEnumerateShortLabel{I}{\textup{(\Roman*)}}
\setlist{listparindent=0.7cm}

\def\lex{\operatorname{lex}}

\def\cl{\operatorname{cl}}
\def\Cell{\operatorname{Cell}}

\def\dim{\operatorname{dim}}

\def\LT{\operatorname{LT}}

\def\ini{\operatorname{in}} 

\def\isom{\cong}

\def\ker{\operatorname{ker}}
\def\KK{{\mathbb K}}

\def\lex{{\operatorname{lex}}}

\def\part{\operatorname{part}}

\def\std{\operatorname{std}} 

\def\trdeg{\operatorname{tr.deg}}

\def\ZZ{{\mathbb Z}}

\newcommand\bdlambda{{\bm \lambda}}

\newcommand\bdmu{{\bm \mu}}

\newcommand\bdP{{\bm P}}

\newcommand\bfp{\mathbf{p}}
\newcommand\bfq{\mathbf{q}}
\newcommand\bfT{\mathbf{T}}

\newcommand\bfx{\mathbf{x}}

\newcommand\bfy{\mathbf{y}}

\newcommand\calA{\mathcal{A}}

\newcommand\calF{\mathcal{F}}
\newcommand\calG{\mathcal{G}}

\newcommand\calJ{\mathcal{J}}

\newcommand\calL{\mathcal{L}}

\newcommand\calP{\mathcal{P}}
\newcommand\calQ{\mathcal{Q}}

\newcommand\calU{\mathcal{U}}
\newcommand\calV{\mathcal{V}}

\newcommand\calX{\mathcal{X}}
\newcommand\calY{\mathcal{Y}}
\newcommand\calZ{\mathcal{Z}}

\def\sA{\mathscr{A}}

\def\XY{\mathcal{XY}}
\def\QP{\mathcal{QP}}

\allowdisplaybreaks
\begin{document}

\title{Binomial edge rings associated to skew Ferrers diagrams}

\author[Kuei-Nuan Lin, Yi-Huang Shen]{Kuei-Nuan Lin and Yi-Huang Shen}

\thanks{\today}



\address{Department of Mathematics, The Penn State University, McKeesport, PA,
15132, USA}
\email{linkn@psu.edu}

\address{School of Mathematical Sciences, University of Science and Technology of China, Hefei, Anhui, 230026, P.R.~China}
\email{yhshen@ustc.edu.cn}

\begin{abstract}
    In this study, we investigate the binomial edge ring associated with the skew Ferrers diagram. By employing \textsc{Sagbi} basis theory, we construct a quadratic Gr\"{o}bner basis for its defining ideal. As an application, we prove that this ring is a Koszul, Cohen--Macaulay, normal domain. Moreover, we precisely determine its Krull dimension.
\end{abstract}

\maketitle

\section{Introduction}

The main object of this work is the {binomial edge ring} associated with a graph $G$. This is the $\KK$-algebra generated by the generators of the {binomial edge ideal} associated with  $G$. In the present work, we will focus on the case where $G$ is 
a bipartite graph that can be represented by a skew Ferrers diagram.  

The concept of {binomial edge ideals} of finite simple graphs was independently introduced in \cite{MR2669070} and \cite{MR2782571}. These ideals are generated by the $2 \times 2$ minors of a $2 \times n$ generic matrix that are associated with the given graph. They exhibit deep connections with algebraic statistics, as discussed in \cite{DESLattice}, and have been extensively studied over the past fifteen years. Recently, Laclair, Mastroeni, McCullough, and Peeva \cite{LMMPeeva} established a complete characterization of the conditions under which the quotient ring of a binomial edge ideal is Koszul.  

Binomial edge ideals, which are generated by maximal minors of a $2 \times n$ generic matrix, are special determinantal ideals. Determinantal ideals play a central role in commutative algebra and algebraic geometry, and they form the core of the theory of Schubert varieties; see, for example, \cite{BV} and \cite{FultonPragaczSchubert}. 
The $\KK$-algebra generated by all maximal minors of a generic matrix is also a foundational object in these areas. As the homogeneous coordinate ring of the Grassmannian variety, its defining ideal is generated by the classical Pl\"{u}cker relations. This algebraic structure serves as a testing ground for investigating various mathematical properties, including homological properties, free resolutions, straightening laws, and aspects of singularity theory.

The defining equations of a $\KK$-algebra generated by a proper subset of maximal minors of a generic matrix remain largely unknown due to the intricate nature of determinantal relationships. Recently, Almousa, Lin, and Liske~\cite{ALL1} studied the special case where the selected maximal minors satisfy the unit interval property, while Higashitani examined binomial edge rings associated with complete bipartite graphs~\cite[Theorem~1.1]{arXiv:2411.07812}.

Complete bipartite graphs exhibit a natural correspondence with rectangular Ferrers diagrams. More generally, {Ferrers bipartite graphs} are derived from {Ferrers diagrams} (also known as {Young diagrams}). The associated edge ideal, termed a {Ferrers ideal}, constitutes a squarefree quadratic monomial ideal possessing a $2$-linear minimal free resolution. Corso and Nagel \cite{CN} conducted a systematic investigation of Ferrers ideal, including their associated toric rings. Subsequently, Corso, Nagel, Petrovi\'c, and Yuen \cite{CNPY} examined the case of skew Ferrers diagrams. Their research elucidated the connections between skew Ferrers diagrams with algebraic statistics. Furthermore, they showed that the $\KK$-algebra generated by the monomial generators of edge ideals associated with skew Ferrers diagrams are Koszul, Cohen--Macaulay, and normal domains.

Therefore, in this study, we aim to extend Higashitani's findings by investigating the binomial edge rings associated with skew Ferrers diagrams. We delineate below the primary objects of this research and summarize the key results obtained, with comprehensive details presented in subsequent sections.

A skew Ferrers diagram  $\boldsymbol{\lambda} / \boldsymbol{\mu}$
can be defined by two tuples (partitions) 
\(\boldsymbol{\lambda} = (\lambda_1,\dots,\lambda_a) \in \mathbb{Z}_{\ge 1}^a\) 
and 
\(\boldsymbol{\mu} = (\mu_1,\dots,\mu_a) \in \mathbb{Z}_{\ge 0}^a\), where
\(
\lambda_1 \ge \cdots \ge \lambda_a\), 
\(\mu_1 \ge \cdots \ge \mu_a\),
and \(\lambda_i \ge \mu_i\) for all \(1 \le i \le a\), with
\(a\) and $b\coloneqq \lambda_1$ being positive integers.  
This diagram \(\boldsymbol{\lambda} / \boldsymbol{\mu}\) naturally induces a bipartite graph $\calG=\calG_{\bdlambda/\bdmu}$ with vertex partition $\{x_1,\dots,x_a\}\sqcup\{y_1,\dots,y_b\}$, such that $\{x_i,y_j\}$ is an edge of $\calG$ if and only if $1\le i\le a$ and $\mu_i< j\le \lambda_i$.

For the bipartite graph \(\mathcal{G}\), we consider the polynomial ring
\[
    S \coloneqq \mathbb{K}[x_1,\ldots,x_a, \ p_1,\ldots,p_b, \ q_1,\ldots,q_a, \ y_1,\ldots,y_b]
\]
over a field \(\mathbb{K}\).  
The \emph{binomial edge ideal} \(J_{\mathcal{G}}\) associated with \(\mathcal{G}\) is generated by the binomials in the set
\[
    G(J_{\mathcal{G}}) \coloneqq 
    \left\{
        f_{i,j} \coloneqq x_i y_j - p_j q_i 
        \ \middle| \ 
        \{x_i, y_j\} \in E(\mathcal{G})
    \right\}.
\]

The central aim of this paper is to establish a \textsc{Sagbi} basis for the associated {binomial edge ring} \(\mathbb{K}[G(J_{\mathcal{G}})] \subseteq S\), and to describe a quadratic Gr\"obner basis for its defining ideal.  
To achieve this goal, we introduce auxiliary polynomials defined as
\[
    f_{i_1,j_1; \, i_2,j_2} 
    \coloneqq 
    f_{i_1,j_1} f_{i_2,j_2} - f_{i_1,j_2} f_{i_2,j_1},
\]
where \(1 \le i_1 < i_2 \le a\) and \(\mu_{i_1} < j_1 < j_2 \le \lambda_{i_2}\).  
The set \(H\), comprising all \(f_{i,j}\) along with \(f_{i_1,j_1; \, i_2,j_2}\), plays a crucial role in our analysis.

\begin{Theorem*}[Main result]
    \begin{enumerate}[label=\textnormal{(\alph*)}]
        \item The set \(H\) forms a \textsc{Sagbi} basis for the binomial edge ring \(\mathbb{K}[G(J_{\mathcal{G}})]\) with respect to an appropriate term order on \(S\).
        \item The defining ideal of \(\mathbb{K}[G(J_{\mathcal{G}})]\) admits a quadratic Gr\"obner basis.
        \item The ring \(\mathbb{K}[G(J_{\mathcal{G}})]\) is a Koszul, Cohen--Macaulay, normal domain.  
            Moreover, if \(\operatorname{char} \mathbb{K} = 0\), it has rational singularities;  
            if \(\operatorname{char} \mathbb{K} > 0\), it is \(F\)-rational.
    \end{enumerate}
\end{Theorem*}

The present work constitutes a natural extension of Higashitani's study \cite{arXiv:2411.07812}, which employed \textsc{Sagbi} theory to characterize the defining ideal of binomial edge rings associated with complete bipartite graphs. Higashitani's proof crucially depended on the isomorphism between the initial algebra of the subalgebra and the Hibi ring of a specific poset. However, this correspondence fails to generalize to the broader context of skew Ferrers diagrams examined in our study.

Instead, we leverage the structure of skew Ferrers diagrams to draw out the relationships among initial monomials within $H$, the carefully selected collection of polynomials. Unlike the Hibi ring case, which benefits from an existing Gr\"{o}bner basis, our framework requires the explicit construction of a reduction process. Through the introduction of an appropriately defined product order, we validate both the Noetherian property of this process and the convergence of all monomials within the same fiber to identical standard monomials. These technical preparations culminate in the proof of our main results. Moreover, Theorem~\ref{thm:dim} provides an explicit computation of the Krull dimension of \(\mathbb{K}[G(J_{\mathcal{G}})]\).

\section{Settings}

Throughout this paper, we consider a fixed skew Ferrers diagram $\Gamma=\bdlambda/\bdmu$ defined by the tuples $\bdlambda=(\lambda_1,\dots,\lambda_a)\in \ZZ_{\ge 1}^a$ with $\lambda_1\ge \cdots \ge \lambda_a$ and $\bdmu=(\mu_1,\dots,\mu_a)\in \ZZ_{\ge 0}^a$ with $\mu_1\ge \cdots \ge \mu_a$. Specifically, $a$ denotes a positive integer, and the condition $\lambda_i\ge \mu_i$ holds for $i=1,2,\dots,a$. 

\begin{Example}
    \label{exam:base_diagram}
    An example with $\bdlambda=(6,5,5,3)$ and $\bdmu=(2,1,0,0)$ is depicted in \Cref{fig:6533}.  For instance, in this diagram, the cells $(2,3)$ and $(4,2)$ are marked with $\times$, while the cells $(1,6)$ and $(3,5)$ are marked with $\triangle$. Note that this style of representation is more consistent with the way entries in matrices are presented, as opposed to the conventional representation of points in a 2-dimensional Euclidean plane.
    \begin{figure}[h]
        \begin{tikzpicture}[x=2\unitlength,y=2\unitlength,inner sep=0pt,>=stealth]
            \draw [line width=0.4mm](50,80) rectangle (60,70);
            \draw [line width=0.4mm](60,80) rectangle (70,70);
            \draw [line width=0.4mm](70,80) rectangle (80,70);
            \draw [line width=0.4mm](80,80) rectangle (90,70);
            \draw [line width=0.4mm](40,70) rectangle (50,60);
            \draw [line width=0.4mm](50,70) rectangle (60,60);
            \draw [line width=0.4mm](60,70) rectangle (70,60);
            \draw [line width=0.4mm](70,70) rectangle (80,60);
            \draw [line width=0.4mm](40,60) rectangle (50,50);
            \draw [line width=0.4mm](50,60) rectangle (60,50);
            \draw [line width=0.4mm](60,60) rectangle (70,50);
            \draw [line width=0.4mm](70,60) rectangle (80,50);
            \draw [line width=0.4mm](30,60) rectangle (40,50);
            \draw [line width=0.4mm](40,50) rectangle (50,40);
            \draw [line width=0.4mm](50,50) rectangle (60,40);
            \draw [line width=0.4mm](30,50) rectangle (40,40);
            \draw (55,65) node {$\times$};
            \draw (45,45) node {$\times$};
            \draw (85,75) node {$\triangle$};
            \draw (75,55) node {$\triangle$};
            \draw [color=gray] (25,75) node {$1$};
            \draw [color=gray] (25,65) node {$2$};
            \draw [color=gray] (25,55) node {$3$};
            \draw [color=gray] (25,45) node {$4$};
            \draw [color=gray] (35,85) node {$1$};
            \draw [color=gray] (45,85) node {$2$};
            \draw [color=gray] (55,85) node {$3$};
            \draw [color=gray] (65,85) node {$4$};
            \draw [color=gray] (75,85) node {$5$};
            \draw [color=gray] (85,85) node {$6$};
        \end{tikzpicture}
        \caption{The skew diagram $\bdlambda/\bdmu$ with $\bdlambda=(6,5,5,3)$ and $\bdmu=(2,1,0,0)$}
        \label{fig:var}
        \label{fig:6533}
    \end{figure}
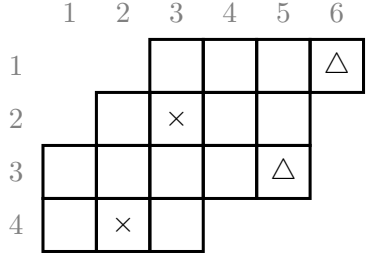
\end{Example}

Let $b\coloneqq \lambda_1$.
Consequently, we obtain an associated bipartite graph $\calG=\calG_{\bdlambda/\bdmu}$ with vertex partition $\{x_1,\dots,x_a\}\sqcup\{y_1,\dots,y_b\}$, such that $\{x_i,y_j\}$ is an edge of $\calG$ if and only if the cell $(i,j)$ belongs to the diagram $\Gamma$, i.e., $1\le i\le a$ and $\mu_i< j\le \lambda_i$.

For the aforementioned bipartite graph $\calG
$, we introduce the polynomial ring $S=\KK[\bfx,\bfp,\bfq,\bfy]$ over a field $\KK$, with variables ordered as
\[
    x_1>x_2>\cdots >x_a> p_1>p_2> \cdots > p_b> q_1>q_2>\cdots>q_a> y_1>y_2>\cdots>y_b.
\]
The lexicographic order on $S$ induced by this variable ordering is denoted by $>_{\lex}$. 
The binomial edge ideal $J_\calG$ of the bipartite graph $\calG$ is generated by the binomials in the finite set
\begin{equation}
    G(J_\calG)\coloneqq \{f_{i,j}\coloneqq \underline{x_i y_j}-p_jq_i\mid \{x_i,y_j\}\in E(\calG)\} 
    \label{eqn:f_ij}
\end{equation}
within the polynomial ring $S$. 
The main theme of this paper is to determine a \textsc{Sagbi} basis for the associated binomial edge ring $\KK[G(J_\calG)]\subseteq S$ and then describe a quadratic G\"{o}bner basis of the binomial edge ring $\KK[G(J_\calG)]$. For this purpose, we also introduce auxiliary polynomials
\begin{equation}
    f_{i_1,j_1;i_2,j_2}\coloneqq f_{i_1,j_1}f_{i_2,j_2}-f_{i_1,j_2}f_{i_2,j_1} = \underline{x_{i_1}p_{j_1}q_{i_2}y_{j_2}} + \cdots,
    \label{eqn:f_iijj}
\end{equation}
where $1\le i_1<i_2\le a$ and $u_{i_1}< j_1<j_2\le \lambda_{i_2}$.

\begin{Remark}
    The set \(G(J_{\mathcal{G}})\) does not constitute a Gr\"obner basis for the ideal it generates with respect to the lexicographic order, since the graph $\mathcal{G}$ is not closed~\cite[Theorem~1.1]{Closed}. Recall that a bipartite graph is closed if and only if it is a line~\cite[Corollary~1.3]{Closed}; however, our bipartite graph $\mathcal{G}$ is not a line in general. 
\end{Remark}

\begin{Notation}
    Let 
    \[
        H\coloneqq\{f_{i,j} \mid (i,j)\in \Gamma \} \cup \{ f_{i_1,j_1;i_2,j_2} \mid 1\le i_1<i_2\le a \text{ and } u_{i_1}< j_1<j_2\le \lambda_{i_2}\}
    \]
    be the set of polynomials defined in Equations \eqref{eqn:f_ij} and \eqref{eqn:f_iijj}. Additionally, let $\LT(H)$ represent the set of leading monomials of the elements in $H$ with respect to $>_{\lex}$. It can be readily verified that the underlined monomials in  \eqref{eqn:f_ij} and \eqref{eqn:f_iijj} are the leading monomials of their respective polynomials. Therefore, we obtain
    \[
        \LT(H)=\{x_iy_j \mid (i,j)\in \Gamma \} \cup \{ x_{i_1}p_{j_1}q_{i_2}y_{j_2}  \mid 1\le i_1<i_2\le a \text{ and } u_{i_1}< j_1<j_2\le \lambda_{i_2}\}.
    \] 
\end{Notation}

\section{Gr\"obner basis of the toric ideal}
\label{sec:GBtoricideal}

In this section, we compute a Gr\"obner basis for the defining ideal of the $\KK$-algebra $\KK[\LT(H)]$. Subsequently, we employ this result to derive a Gr\"obner basis for the defining ideal of the binomial edge ring $\KK[G(J_\calG)]$ in \Cref{sec:Sagbi}. 

To describe the defining ideals, we first introduce the presentation maps. Let $R=R_{\Gamma}\coloneqq \KK[\bfT]$ be the polynomial ring in the variables 
\[
    \{T_{i,j}\mid 1\le i\le a, 1\le j \le \lambda_i\}\sqcup \{T_{i_1,j_1;i_2,j_2}\mid 1\le i_1<i_2\le a,\mu_i< j_1<j_2\le \lambda_{i_2}\}
\]
over a field $\KK$. We begin with the presentation
\[
    \varphi:R\to \KK[H], \qquad T_{i,j}\mapsto f_{i,j}, \quad T_{i_1,j_1;i_2,j_2}\mapsto f_{i_1,j_1;i_2,j_2},
\]
along with the ``deformation'' presentation
\[
    \varphi^* :R\to 
    \KK[\LT(H)],
    \qquad T_{i,j}\mapsto x_iy_j, \quad T_{i_1,j_1;i_2,j_2}\mapsto x_{i_1}p_{j_1}q_{i_2}y_{j_2}.
\]

To better capture the underlying combinatorial structure, we require a broader framework.

\begin{Definition}
    \label{def:ambient}
    \begin{enumerate}[a]
        \item Let $\overline{\Gamma}\coloneqq \overline{\bdlambda}/\overline{\bdmu}$ be the rectangular Ferrers diagram, where $\overline{\bdlambda}=(b,b,\dots,b)$ and $\overline{\bdmu}=(0,0,\dots,0)$ represent two tuples in $\ZZ^a$. In this configuration, the associated graph $\overline{\calG}\coloneqq \calG_{\overline{\bdlambda}/\overline{\bdmu}}$ is the complete bipartite graph with the vertex partition $\{x_1,\dots,x_a\}\sqcup\{y_1,\dots,y_b\}$. 
        \item Related, let $\overline{R}\coloneqq R_{\overline{\Gamma}}$ be the polynomial ring over $\KK$ with variables 
            \[
                \{T_{i,j}\mid 1\le i\le a, 1\le j \le b\}\sqcup \{T_{i_1,j_1;i_2,j_2}\mid 1\le i_1<i_2\le a,1\le j_1<j_2\le b\}.
            \]
            For notational convenience, we denote this variable set as $\{T_\ell\mid \ell\in \calL\}$, where the index $\ell\in \calL$ if and only if $T_{\ell}$ is a variable in $\overline{R}$.
        \item For $T_{i,j}\in \overline{R}$, we define $\Cell(T_{i,j})\coloneqq \{(i,j)\}\subseteq \overline{\Gamma}$, while for $T_{i_1,j_1;i_2,j_2}$, we define $\Cell(T_{i_1,j_1;i_2,j_2})\coloneqq \{(i_1,j_1),(i_1,j_2),(i_2,j_1),(i_2,j_2)\}\subseteq \overline{\Gamma}$. More generally, for each polynomial
            \[
                f=\sum_{i=1}^{r} k_i \prod_{\ell\in \calL} T_\ell^{e_{i,\ell}} \in \overline{R},
            \]
            where $k_i\in \KK\setminus\{0\}$ and $e_{i,\ell}\in \ZZ_{\ge 0}$, we define
            \[
                \Cell(f)\coloneqq \bigcup_{i=1}^{r} \bigcup_{\substack{\ell\in \calL,\\ e_{i,\ell}>0}} \Cell(T_\ell)\subseteq \overline{\Gamma},
            \]
            and call it the \emph{cell set} of $f$.
    \end{enumerate} 
\end{Definition}

Our proof strategy utilizes the combinatorial structure of the skew Ferrers diagram to establish relations among the leading monomials of the polynomials in $H$. More precisely, each monomial $x_i y_j$ corresponds to a cell in ${\Gamma}$, while each monomial $x_{i_1} p_{j_1} q_{i_2} y_{j_2}$ corresponds to two anti-diagonal cells of a rectangular region in ${\Gamma}$. To formalize this approach, we first introduce key definitions pertaining to the positional relationships between cells in $\overline{\Gamma}$.

\begin{Definition}
    \label{rmk:NE_SW}
    Let $(i,j)$ and $(i',j')$ be two cells in the diagram $\overline{\Gamma}$.
    \begin{enumerate}[a]
        \item If \begin{equation}
                i\le i' \quad \text{and} \quad j\ge j',
                \label{eqn:NE}
            \end{equation}
            then we say that $(i,j)$ is \emph{weakly northeast} of $(i',j')$. We also say that the ordered pair $((i,j),(i',j'))$ is \emph{weakly NE-compatible}, and the unordered pair $\{(i,j),(i',j')\}$ is \emph{weakly NE-SW-compatible}.  If both inequalities in \eqref{eqn:NE} are strict, then we have the \emph{strict} version of concepts. We also have similar and obvious descriptions for the other directions, which we won't need to repeat here.

        \item When the cell $(i,j)$ is {weakly northeast} of $(i',j')$, and all four cells $(i,j)$, $(i',j')$, $(i',j)$ and $(i,j')$ are contained within $\Gamma$, we define the ordered pair $((i,j),(i',j'))$ as a \emph{legitimate} pair with respect to $\Gamma$. 
    \end{enumerate} 
\end{Definition}

The legitimacy of the ordered pair is employed to distinguish $\overline{R}$ from $R$. The effective utilization of this approach is based on the following observations:

\begin{Remark}
    \label{rmk:diagram_and_variables}
    \begin{enumerate}[a]
        \item For every polynomial $f\in \overline{R}$, we have $f\in R$ if and only if $\Cell(f)\in \Gamma$.
        \item \label{rmk:diagram_and_variables_a}
            Suppose that $(i,j)$ and $(i',j')$ are two cells in $\Gamma$. Let $\overline{\Gamma}$ be the ambient diagram in \Cref{def:ambient}. Whence, $(i,j)$ and $(i',j')$ determine a rectangular region on $\overline{\Gamma}$ as follows.
            \begin{enumerate}[i]
                \item If $(i,j)$ and $(i',j')$ are weakly NW-SE-compatible, these two cells are the diagonal corners of this region.  Since $\bdlambda/\bdmu$ is a skew diagram, it can be readily verified that all cells within this region belong to $\Gamma$. 
                \item Conversely, if $(i,j)$ and $(i',j')$ are weakly NE-SW-compatible, they form the anti-diagonal corners of this region. The diagonal corners of this region are necessarily the cells $(i,j')$ and $(i',j)$. However, the cells in this region are not necessarily contained in $\Gamma$. 
            \end{enumerate} 
            \begin{figure}[tbp]
                \begin{tikzpicture}[x=2\unitlength,y=2\unitlength,inner sep=0pt]
                    \fill [pattern=crosshatch, pattern color=gray!50](40,40) rectangle (60,70);
                    \fill [pattern=crosshatch dots, pattern color=gray!50](70,50) rectangle (90,80);
                    \draw [line width=0.4mm](50,80) rectangle (60,70);
                    \draw [line width=0.4mm](60,80) rectangle (70,70);
                    \draw [line width=0.4mm](70,80) rectangle (80,70);
                    \draw [line width=0.4mm](80,80) rectangle (90,70);
                    \draw [line width=0.4mm](40,70) rectangle (50,60);
                    \draw [line width=0.4mm](50,70) rectangle (60,60);
                    \draw [line width=0.4mm](60,70) rectangle (70,60);
                    \draw [line width=0.4mm](70,70) rectangle (80,60);
                    \draw [line width=0.4mm](50,60) rectangle (60,50);
                    \draw [line width=0.4mm](60,60) rectangle (70,50);
                    \draw [line width=0.4mm](70,60) rectangle (80,50);
                    \draw [line width=0.4mm](50,50) rectangle (60,40);
                    \draw [line width=0.4mm](40,60) rectangle (50,50);
                    \draw [line width=0.4mm](30,60) rectangle (40,50);
                    \draw [line width=0.4mm](40,50) rectangle (50,40);
                    \draw [line width=0.4mm](30,40) rectangle (50,60); 
                \end{tikzpicture}
                \caption{Rectangular regions}
                \label{fig:region}
            \end{figure}
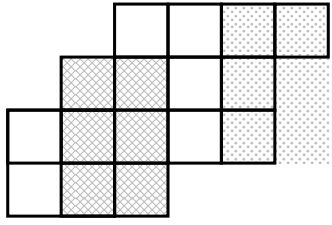

        \item \label{rmk:diagram_and_variables_c}
            The variables of $R$ can be visualized by the diagram $\Gamma$ as follows. 
            \begin{enumerate}[i]
                \item The cell $(i,j)$ in the diagram $\Gamma$ corresponds bijectively to the variable $T_{i,j}$ in $R$. 
                \item The legitimate strictly NE-compatible pair $((i,j),(i',j'))$ in $\Gamma$ corresponds bijectively to the variable $T_{i,j';i',j}$ in $R$. More precisely, this legitimate pair creates a rectangular region in $\Gamma$, in which they are anti-diagonal corners. The coordinates of the diagonal corners of this region emerge in turn as the subscripts of the $T$ variable.
            \end{enumerate} 
    \end{enumerate} 
\end{Remark}

\begin{Example}
    In the skew Ferrers diagram $\Gamma$ depicted in \Cref{fig:var}, the ordered pair $((2,3),(4,2))$ (marked with $\times$ in the figure) is a legitimate pair. The rectangular region determined by the cells $(2,3)$ and $(4,2)$ is highlighted by $
    \begin{tikzpicture}[x=1\unitlength,y=1\unitlength,inner sep=0pt]
        \filldraw [pattern=crosshatch, pattern color=gray!80](0,0) rectangle (8,8);
    \end{tikzpicture}
    $ in \Cref{fig:region}, with all constituent cells remaining entirely within $\Gamma$. It is also the rectangular region determined by the cells $(2,2)$ and $(4,3)$.

    At the same time, the ordered pair $((1,6),(3,5))$ (marked with $\triangle$ in the figure) is not legitimate. The rectangle region determined by the cells $(1,6)$ and $(3,5)$ is highlighted by
    $
    \scalebox{0.5}{
        \begin{tikzpicture}[x=2\unitlength,y=2\unitlength,inner sep=0pt]
            \filldraw [pattern=crosshatch dots, pattern color=gray!80](0,0) rectangle (8,8);
    \end{tikzpicture}}
    $ in \Cref{fig:region}. Notably, this region includes cells $(2,6)$ and $(3,6)$, which lie outside the boundary of $\Gamma$. 
\end{Example}

We plan to apply \cite[Theorem 3.12]{Sturmfels} to determine a Gr\"{o}bner basis of the defining ideal of $\KK[\LT(H)]$. To this end, we first introduce the reduction relations.

\begin{Definition}
    \label{def:GB_reduction}
    We introduce three types of quadratic reductions on the monomials in $R$ as follows.  Suppose that $(i_1,j_1)$, $(i_2,j_2)$, $(i_1',j_1')$, and $(i_2',j_2')$ are cells in $\Gamma$. Firstly, we have
    \begin{align}
        T_{i_1,j_1}T_{i_1',j_1'}&\quad \leadsto \quad T_{\min\{i_1,i_1'\},\max\{j_1,j_1'\}}T_{\max\{i_1,i_1'\},\min\{j_1,j_1'\}}, \tag{I}\label{eqn:reduction_22}\\
        \intertext{and}
        T_{i_1,j_1;i_2,j_2}T_{i_1',j_1';i_2',j_2'}&\quad\leadsto \quad 
        T_{\min\{i_1,i_1'\},\max\{j_1,j_1'\};\min\{i_2,i_2'\},\max\{j_2,j_2'\}} \notag\\
        & \qquad \qquad \qquad \times T_{\max\{i_1,i_1'\},\min\{j_1,j_1'\};\max\{i_2,i_2'\},\min\{j_2,j_2'\}}.\tag{I\!I} \label{eqn:reduction_44}
    \end{align}
    In addition, if $T_{\min\{i_1,i_1'\},j_1';i_2',\max\{j_2,j_2'\}}\in R$, then we have
    \begin{equation}
        T_{i_1',j_1';i_2',j_2'}T_{i_1,j_1}\quad\leadsto \quad
        T_{\min\{i_1,i_1'\},j_1';i_2',\max\{j_1,j_2'\} T_{\max\{i_1,i_1'\},\min\{j_1,j_2'\}}}.
        \tag{I\!I\!I$_1$}\label{eqn:reduction_24}
    \end{equation}
    Otherwise, $T_{\min\{i_1,i_1'\},j_1';i_2',\max\{j_2,j_2'\}}\notin R$, and we will have instead
    \begin{equation}
        T_{i_1',j_1';i_2',j_2'} T_{i_1,j_1}\quad\leadsto \quad
        T_{\max\{i_1,i_1'\},j_1',i_2',\min\{j_1,j_2'\}}T_{\min\{i_1,i_1'\},\max\{j_1,j_2'\}}.
        \tag{I\!I\!I$_2$}\label{eqn:reduction_42}
    \end{equation}

    Let $\calF$ be the collection of all these reductions, \emph{excluding} the trivial identical ones.
\end{Definition}

The reductions presented in \Cref{def:GB_reduction} will be employed to construct the expected Gr\"obner basis. Prior to implementing these reductions, it is essential to verify their validity and effects.

\begin{Lemma}
    [Validity of reductions]
    \label{rmk:validity}
    The three types of reductions in \Cref{def:GB_reduction} are well-defined on the set of monomials of $R$.
\end{Lemma}
\begin{proof}
    We will examine the three types of reductions individually in the following.
    \begin{enumerate}[align=left, leftmargin=*]
        \item [\eqref{eqn:reduction_22}:] 
            The rectangular region determined by $(i_1,j_1)$ and $(i_1',j_1')$ is precisely the rectangular region determined by $(\min\{i_1,i_1'\},\max\{j_1,j_1'\})$ and $(\max\{i_1,i_1'\},\min\{j_1,j_1'\})$. Furthermore, the cell $(\min\{i_1,i_1'\},\max\{j_1,j_1'\})$ is always the northeast corner of this region, while $(\max\{i_1,i_1'\},\min\{j_1,j_1'\})$ is always the southwest corner. 
                If $(i_1,j_1)$ and $(i_1',j_1')$ are NW–SE compatible, then their presence in the diagram $\Gamma$ implies, according to \Cref{rmk:diagram_and_variables}\ref{rmk:diagram_and_variables_a}, that both $(\min\{i_1,i_1'\},\max\{j_1,j_1'\})$ and $ (\max\{i_1,i_1'\},\min\{j_1,j_1'\})$ must belong to $\Gamma$.
            If instead $(i_1,j_1)$ and $(i_1',j_1')$ are NE-SW-compatible, then 
            \[
                \{(i_1,j_1),(i_1',j_1)\}=\{(\min\{i_1,i_1'\},\max\{j_1,j_1'\}), (\max\{i_1,i_1'\},\min\{j_1,j_1'\})\}.
            \]
            Consequently, this case is trivial, as no reduction is applied.
            We still have 
            \[
                (\min\{i_1,i_1'\},\max\{j_1,j_1'\}), (\max\{i_1,i_1'\},\min\{j_1,j_1'\})\in \Gamma.
            \]
            Therefore, the reduction in \eqref{eqn:reduction_22} is valid.
        \item [\eqref{eqn:reduction_44}:] In a similar vein, we can show that 
            \begin{align*}
                &(\min\{i_1,i_1'\},\max\{j_1,j_1'\}),\quad (\min\{i_2,i_2'\},\max\{j_2,j_2'\}),\\
                &\quad (\max\{i_1,i_1'\},\min\{j_1,j_1'\}),\quad (\max\{i_2,i_2'\},\min\{j_2,j_2'\})\in \Gamma.
            \end{align*}
            Furthermore, we have the following elementary fact:
            \[
                \text{if $a<b$ and $c<d$, then $\min(a,c)<\min(b,d)$ and $\max\{a,c\}<\max\{b,d\}$.}
            \]
            Since $i_1<i_2$, $j_1<j_2$, $i_1'<i_2'$, and $j_1'<j_2'$, we have consequently
            \[
                \min\{i_1,i_1'\}<\min\{i_2,i_2'\}, \quad \max\{j_1,j_1'\}<\max\{j_2,j_2'\},
            \]
            and
            \[
                \max\{i_1,i_1'\}<\max\{i_2,i_2'\}, \quad \min\{j_1,j_1'\}<\min\{j_2,j_2'\}.
            \]
            Therefore, the reduction in \eqref{eqn:reduction_44} is valid.

        \item[\eqref{eqn:reduction_24}+\eqref{eqn:reduction_42}:] 
            Since both $(i_1,j_1)$ and $(i_1',j_2')$ belong to $\Gamma$, it follows directly from the argument for the reduction in \eqref{eqn:reduction_22} that both $T_{\min\{i_1,i_1'\},\max\{j_1,j_2'\}}$ and $T_{\max\{i_1,i_1'\},\min\{j_1,j_2'\}}$ are valid variables in $R$. 

            \begin{enumerate}[a]
                \item Suppose that $T_{\min\{i_1,i_1'\},j_1';i_2',\max\{j_1,j_2'\}}\in R$. Clearly, the reduction in \eqref{eqn:reduction_24} is valid in this case.

                \item Suppose instead that $T_{\min\{i_1,i_1'\},j_1';i_2',\max\{j_1,j_2'\}}\notin R$. Notice that there is a rectangular region in $\Gamma$, which is determined by $T_{i_1',j_1';i_2',j_2'}$ as explained in \Cref{rmk:diagram_and_variables}\ref{rmk:diagram_and_variables_c}. Denote this region by $Z_1$. The northeast corner $(i_1',j_2')$ of $Z_1$ and the cell $(i_1,j_1)$ lead to another rectangular region $Z_2$. One can use the northeast corner of $Z_2$ and the southwest corner $(i_2',j_1')$ of $Z_1$ to create a rectangular region in $\overline{\Gamma}$, which will be denoted by $Z_3$. The assumption that the variable $T_{\min\{i_1,i_1'\},j_1';i_2',\max\{j_1,j_2'\}}\notin R$ amounts to saying that the region $Z_3$ is not contained in $\Gamma$. Since the northeast corner of $Z_3$ is the northeast corner of $Z_2$ and the southwest corner of $Z_3$ is the southwest corner of $Z_1$, it amounts to checking the (non)-presence of the diagonal corners of $Z_3$.

                    Depending on the relative positions of $(i_1,j_1)$ to the cells $(i_1',j_2')$ and $(i_2',j_1')$, there are $9$ subcases to scrutinize; see also \Cref{fig:validity}. Given the observation in the previous paragraph, it is straightforward to verify using \Cref{rmk:diagram_and_variables}\ref{rmk:diagram_and_variables_a} that, since $T_{\min\{i_1,i_1'\},j_1';i_2',\max\{j_1,j_2'\}}\notin R$, the cell \emph{$(i_1,j_1)$ lies strictly northeast of $(i_2',j_1')$, and is not weakly southwest of $(i_1',j_2')$}. The potential positions of $(i_1,j_1)$ are depicted in the dotted regions in \Cref{fig:validity}.
                    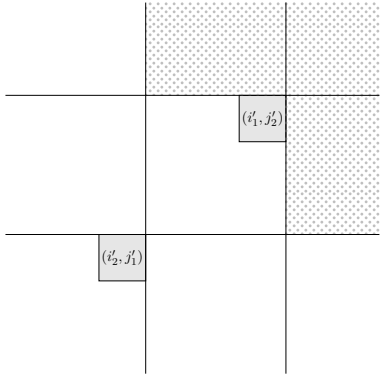
\begin{figure}[tbp]
                        \scalebox{0.5}{
                            \begin{tikzpicture}[x=3.5\unitlength,y=3.5\unitlength,inner sep=0pt]
                                \clip (10,-0) rectangle (90,80);
                                \filldraw [fill=gray!20] (60,60) rectangle (70,50);
                                \filldraw [fill=gray!20] (30,20) rectangle (40,30);
                                \fill [pattern = crosshatch dots, pattern color=gray!50] (70,30) rectangle (90,80);
                                \fill [pattern = crosshatch dots, pattern color=gray!50] (40,60) rectangle (70,80);
                                \draw (40,80) -- (40,-0);
                                \draw (70,80) -- (70,0);
                                \draw (10,60) -- (90,60);
                                \draw (10,30) -- (90,30);
                                \draw (65,55) node {$(i_1',j_2')$};
                                \draw (35,25) node {$(i_2',j_1')$};
                        \end{tikzpicture}}
                        \caption{Condition for $T_{\min\{i_1,i_1'\},j_1';i_2',\max\{j_1,j_2'\}}\notin R$}
                        \label{fig:validity}
                    \end{figure}

                    Under the last assumption on $(i_1,j_1)$, we use the southwest corner of $Z_2$ and the southwest corner of $Z_1$ to create a rectangular region $Z_4$. The assumption on $(i_1,j_1)$ will guarantee that $Z_4$, being a subregion of $Z_1$, is contained in $\Gamma$. Furthermore, the southwest corner of $Z_2$ is the northeast corner of $Z_4$, and the southwest corner of $Z_1$ is the southwest corner of $Z_4$. Whence, $T_{\max\{i_1,i_1'\},j_1',i_2',\min\{j_1,j_2'\}}$ is a valid variable in $R$, whose associated rectangular region is $Z_4$.  Therefore, the reduction in \eqref{eqn:reduction_42} is valid in this case.
                    \qedhere
            \end{enumerate} 
    \end{enumerate} 
\end{proof}

To gain a clearer understanding of the three types of reductions in \Cref{def:GB_reduction}, we elaborate on their combinatorial effects in what follows. For conciseness, we introduce the following notations.

\begin{Notation}
    \label{notn:4}
    \begin{enumerate}[a]
        \item We also denote $T_{i,j}$ as $T_{i,-\infty;+\infty,j}$, thereby ensuring that every variable in $R$ formally bears a subscript in the form of a $4$-tuple.  To maintain consistency with the homomorphism $\varphi^{*}$ introduced at the beginning of \Cref{sec:GBtoricideal}, we also denote $p_{-\infty}$ and $q_{+\infty}$ as the identity element $1$ in $S$.

        \item Given $E\coloneqq (i_1,j_1;i_2,j_2)\in (\ZZ\cup\{\pm\infty\})^4$,  let
            \[
                \mathcal{XY}_{E} \coloneqq (i_1,j_2) \qquad\text{and}\qquad \mathcal{QP}_{E} \coloneqq (i_2,j_1). 
            \]
            These definitions are established in accordance with the construction of $\varphi^{*}$.
            If $T_{E}$ is a variable in $R$, we denote these as $\mathcal{XY}_{T_E}$ and $\mathcal{QP}_{T_E}$ respectively. In particular, $\QP_{T_{i_1,j_2}}=(+\infty,-\infty)$, which lies strictly southwest of every cell on $\overline{\Gamma}$.
    \end{enumerate} 
\end{Notation}

\begin{Remark}
    [Effect of reductions]
    \label{rmk:GB_reduction}
    Using the notations in \Cref{notn:4}, for $T_{A}\coloneqq T_{i_1,j_1;i_2,j_2}$ and $T_{B}\coloneqq T_{i_1',j_1';i_2',j_2'}\in R$, the reductions \eqref{eqn:reduction_22}, \eqref{eqn:reduction_44}, and \eqref{eqn:reduction_24} take the formal form
    \[
        T_{A}T_{B} \quad \leadsto \quad T_{A'} T_{B'},
    \]
    where
    \begin{align*}
        A'\coloneqq (\min\{i_1,i_1'\},\max\{j_1,j_1'\},\min\{i_2,i_2'\},\max\{j_2,j_2'\})
        \intertext{and} 
        B'\coloneqq (\max\{i_1,i_1'\},\min\{j_1,j_1'\},\max\{i_2,i_2'\},\min\{j_2,j_2'\}).
    \end{align*}
    The following observations are immediate.
    \begin{enumerate}[i]
        \item It follows from \Cref{rmk:validity} that the four cells $\mathcal{XY}_{A'}$, $\mathcal{QP}_{A'}$, $\mathcal{XY}_{B'}$, and $\mathcal{QP}_{B'}$ all lie in $\Gamma\cup \{(+\infty,-\infty)\}$. These cells are uniquely determined by $\varphi^*(T_{A} T_{B})$.
        \item The rectangular region determined by $\XY_A$ and $\XY_B$ coincides exactly with that determined by $\XY_{A'}$ and $\XY_{B'}$. Furthermore, in the reduction \eqref{eqn:reduction_44}, the rectangular region determined by $\QP_A$ and $\QP_B$ is identical to that determined by $\QP_{A'}$ and $\QP_{B'}$.
        \item The cell $\mathcal{XY}_{A'}$ lies weakly northeast of $\mathcal{XY}_{B'}$, and the cell $\mathcal{QP}_{A'}$ lies weakly northeast of $\mathcal{QP}_{B'}$. Meanwhile, the cell $\mathcal{XY}_{A'}$ lies strictly northeast of $\mathcal{QP}_{A'}$, and the cell $\mathcal{XY}_{B'}$ lies strictly northeast of $\mathcal{QP}_{B'}$, as established by the earlier validity argument.
    \end{enumerate} 

    Regarding the remaining reduction in \eqref{eqn:reduction_42}, let us define $T_A\coloneqq T_{i_1,j_1}$, $T_B\coloneqq T_{i_1',j_1';i_2',j_2'}$, $T_{A''} \coloneqq T_{\min\{i_1,i_1'\},\max\{j_1,j_2'\}}$, and $T_{B''}\coloneqq T_{\max\{i_1,i_1'\},j_1',i_2',\min\{j_1,j_2'\}}$ within it. We present the following observations instead.
    \begin{enumerate}[i$'$]
        \item It follows from \Cref{rmk:validity} that the three cells $\mathcal{XY}_{A''}$,  $\mathcal{XY}_{B''}$, and $\mathcal{QP}_{B''}$ all lie in $\Gamma$.  These three cells are uniquely determined by $\varphi^*(T_{A} T_{B})$. Furthermore, formally, we have $\mathcal{QP}_{A''}=(+\infty,-\infty)$.
        \item The rectangular region determined by $\XY_A$ and $\XY_B$ coincides exactly with that determined by $\XY_{A''}$ and $\XY_{B''}$. 
        \item The cell $\mathcal{XY}_{A''}$ lies weakly northeast of $\mathcal{XY}_{B''}$, and the cell $\mathcal{QP}_{B''}$ is exactly $\QP_B$. Additionally, the cell $\mathcal{XY}_{B''}$ lies strictly northeast of $\mathcal{QP}_{B''}$, as established by the earlier validity argument.
    \end{enumerate}
\end{Remark}

To establish the result regarding the Gr\"obner basis, we still need the final piece of the puzzle that guarantees the convergence of the reduction process.

\begin{Definition}\label{def:chi}
    For any two cells $A,B\in \ZZ^2\cup\{(+\infty,-\infty)\}$, we define $\chi(A,B)$ as $0$ if $A$ and $B$ are weakly NE-SW-compatible, and as $1$ otherwise. 
    Notably, following the visual convention established in \Cref{exam:base_diagram}, we consider $(+\infty,-\infty)$ to be weakly NE-SW-compatible with all cells in $\ZZ^2\cup\{(+\infty,-\infty)\}$. 
\end{Definition}

The following lemma quantitatively exemplifies the effect of the reduction process through the newly defined $\chi$-function among three cells in $\Gamma$.

\begin{Lemma}
    Let $A$ and $B$ be two distinct cells in $\ZZ^2$, where $A$ is positioned strictly to the northwest of $B$. These two cells uniquely determine a rectangular region with its northeast corner designated as $A'$ and its southwest corner as $B'$. Then, for any $C\in \ZZ^2\cup\{(+\infty,-\infty)\}$, the following inequality holds: 
    \begin{equation}
        \chi(A,C)+\chi(B,C)\ge \chi(A',C)+\chi(B',C). 
        \label{eqn:compute_chi}
    \end{equation}
\end{Lemma}

\begin{proof}
    The inequality in \eqref{eqn:compute_chi} can be verified using elementary arguments as follows. In the left-hand subfigure of \Cref{fig:sum_of_chi}, the rectangular region determined by $A$ and $B$ is illustrated in gray. With respect to cell $C$, there are roughly $25$ subregions based on its relative position to the four corners. In the right-hand subfigure of \Cref{fig:sum_of_chi}, for each subregion, we mark how the sum of the $\chi$-values changes from the left-hand side to the right-hand side of Inequality \eqref{eqn:compute_chi}. In particular, we observe that this inequality holds in all cases.
    \begin{figure}[tbp]
        \scalebox{0.5}{
            \begin{tikzpicture}[x=3.5\unitlength,y=3.5\unitlength,inner sep=0pt]
                \clip (10,-0) rectangle (90,80);
                \fill[fill=gray!20] (30,60) rectangle (70,20);
                \draw (30,80) -- (30,-0);
                \draw (40,80) -- (40,-0);
                \draw (60,80) -- (60,0);
                \draw (70,80) -- (70,0);
                \draw (10,60) -- (90,60);
                \draw (10,50) -- (90,50);
                \draw (10,30) -- (90,30);
                \draw (10,20) -- (90,20);
                \draw (35,55) node {$A$};
                \draw (65,55) node {$A'$};
                \draw (35,25) node {$B'$};
                \draw (65,25) node {$B$};
        \end{tikzpicture}}
        \qquad
        \scalebox{0.5}{
            \begin{tikzpicture}[x=3.5\unitlength,y=3.5\unitlength,inner sep=0pt]
                \clip (10,-0) rectangle (90,80);
                \draw (30,80) -- (30,-0);
                \draw (40,80) -- (40,-0);
                \draw (60,80) -- (60,0);
                \draw (70,80) -- (70,0);
                \draw (10,60) -- (90,60);
                \draw (10,50) -- (90,50);
                \draw (10,30) -- (90,30);
                \draw (10,20) -- (90,20);
                \draw (20,70) node {$2\to 2$};
                \draw (20,55) node {$1\to 1$};
                \draw (20,40) node {$1\to 1$};
                \draw (20,25) node {$0\to 0$};
                \draw (20,10) node {$0\to 0$};
                \draw (80,10) node {$2\to 2$};
                \draw (80,40) node {$1\to 1$};
                \draw (80,25) node {$1\to 1$};
                \draw (80,55) node {$0\to 0$};
                \draw (80,70) node {$0\to 0$};
                \draw (50,10) node {$1\to 1$};
                \draw (50,70) node {$1\to 1$};
                \draw (50,55) node {$1\to 0$};
                \draw (35,55) node {$1\to 0$};
                \draw (35,40) node {$1\to 0$};
                \draw (65,40) node {$1\to 0$};
                \draw (50,25) node {$1\to 0$};
                \draw (65,10) node {$1\to 1$};
                \draw (35,70) node {$1\to 1$};
                \draw (35,10) node {$0\to 0$};
                \draw (65,55) node {$0\to 0$};
                \draw (65,25) node {$1\to 0$};
                \draw (35,25) node {$0\to 0$};
                \draw (65,70) node {$0\to 0$};
                \draw (50,40) node {$2\to 0$};
        \end{tikzpicture}}
        \caption{Sum of $\chi$-values}
        \label{fig:sum_of_chi}
    \end{figure}
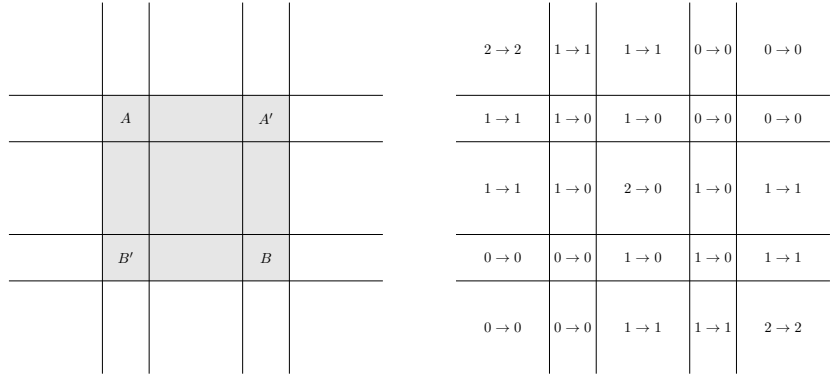
\end{proof}

It is now appropriate to establish the convergence of the reduction process.

\begin{Proposition}
    \label{rmk:StandardMonomial}
    The reduction relation module the $\calF$ in \Cref{def:GB_reduction} is Noetherian, i.e., every sequence of reductions modulo $\calF$ terminates. 
\end{Proposition}

\begin{proof}
    First, we introduce a coarse order on the set of monomials in $R$ as follows. Given a monomial $g=T_{A_1}\cdots T_{A_k}$ of degree $k$ in $R$, we define
    \[
        \chi(g)\coloneqq \sum_{1\le i<j\le k} (\chi(\mathcal{XY}_{T_{A_i}},\mathcal{XY}_{T_{A_j}})+\chi(\mathcal{QP}_{T_{A_i}},\mathcal{QP}_{T_{A_j}})),
    \]
    using the function $\chi$ defined in \Cref{def:chi}.   If applying any of the four reductions in \Cref{def:GB_reduction} to $g$ yields a new monomial $g'$, we claim that $\chi(g)\ge \chi(g')$. For $k=2$, this claim holds by virtue of the observations in items (iii) and (iii)$'$ of \Cref{rmk:GB_reduction}. For $k\ge 3$, it suffices to additionally use the inequality in \eqref{eqn:compute_chi}. 

    The coarse order defined by $\chi$ shall also be denoted by $\chi$, by abuse of notation. It is far from being a 
    total
    order. Therefore, we need to refine it to achieve this end. Given two variables $T_{i_1,j_1;i_2,j_2}$ and $T_{i_1',j_1';i_2',j_2'}$ in $R$ (following the convention in \Cref{rmk:GB_reduction}), we say that $T_{i_1,j_1;i_2,j_2}>T_{i_1',j_1';i_2',j_2'}$ if the first nonzero component of $(i_1-i_1',i_2-i_2',j_1-j_1',j_2-j_2')$ 
    is negative.  We have a graded reverse lexicographic order $\tau$ on $R$ with respect to this ordering of variables.

    Let $\chi\#\tau$ be the product order of $\chi$ and the graded reverse lexicographic order $\tau$. For each reduction in \Cref{def:GB_reduction}, it is straightforward to verify that the quadratic monomial before the reduction is bigger with respect to $\chi\#\tau$ than that after the reduction. Since the order $\tau$ is a total order on the monomials of $R$ and
    the values of $\chi$ are non-negative, it follows directly from \cite[Theorem 3.12]{Sturmfels} that the reduction relation modulo $\calF$ is Noetherian.
\end{proof}

We now formally introduce the binomial relations that constitute a quadratic Gr\"{o}bner basis for $\ker (\varphi^\ast)$, which is the defining ideal of the algebra $\KK[\LT(H)]$.

\begin{Definition}
    \begin{enumerate}[a]
        \item The following binomials correspond to the reductions in \Cref{def:GB_reduction}:
            \begin{align}
                T_{i_1,j_1}T_{i_1',j_1'} &- T_{\min\{i_1,i_1'\},\max\{j_1,j_1'\}}T_{\max\{i_1,i_1'\},\min\{j_1,j_1'\}}, \tag{I}\\
                T_{i_1,j_1;i_2,j_2}T_{i_1',j_1';i_2',j_2'}&-
                T_{\min\{i_1,i_1'\},\max\{j_1,j_1'\};\min\{i_2,i_2'\},\max\{j_2,j_2'\}} \notag\\
                & \qquad \qquad \qquad \times T_{\max\{i_1,i_1'\},\min\{j_1,j_1'\};\max\{i_2,i_2'\},\min\{j_2,j_2'\}},\tag{I\!I}  \\
                T_{i_1',j_1';i_2',j_2'}T_{i_1,j_1}&-
                T_{\min\{i_1,i_1'\},j_1';i_2',\max\{j_1,j_2'\}} T_{\max\{i_1,i_1'\},\min\{j_1,j_2'\}}, \tag{I\!I\!I$_1$}\\
                T_{i_1',j_1';i_2',j_2'}T_{i_1,j_1}&-
                T_{\max\{i_1,i_1'\},j_1',i_2',\min\{j_1,j_2'\}} T_{\min\{i_1,i_1'\},\max\{j_1,j_2'\}}. \tag{I\!I\!I$_2$}
            \end{align}

            By abuse of notation, we also let $\calF$ be the collection of all these binomials, \emph{excluding} the trivial ones. 
        \item According to \cite[Theorem 3.12]{Sturmfels} and \Cref{rmk:StandardMonomial}, the binomials in $\calF$ form a Gr\"obner basis for the ideal they generate with respect to a suitable monomial ordering on $R$. Consequently, any monomial $g$ in $R$ can be reduced by $\calF$ to its normal form. Given that the polynomials in $\calF$ are binomial, this normal form remains a monomial, known as a \emph{standard monomial}. We denote this monomial by $\std(g)$.
        \item Let $\calA$ be an assignment for the monomials in $R$ (such as assignments to suitable monomial, integer, multiset, or other mathematical objects). The assignment $\calA$ is called \emph{fiber-invariant} (with respect to $\varphi^*$) if for any monomials $g$ and $\widetilde{g}$ in $R$ satisfying $\varphi^*(g)=\varphi^*(\widetilde{g})$, the condition $\calA(g)=\calA(\widetilde{g})$ holds. By abuse of notation, we also refer to $\calA(g)$ as being fiber-invariant  when this condition is satisfied.
    \end{enumerate} 
\end{Definition}

Hereafter, we will adopt the notation and conventions established in \Cref{notn:4}. Furthermore, for any given quadruple $E\coloneqq (i_1,j_1;i_2,j_2)\in (\ZZ\cup\{\pm\infty\})^4$, we define 
\[
    \mathcal{X}_{E} \coloneqq i_1, \quad 
    \mathcal{Y}_{E} \coloneqq j_2, \quad 
    \mathcal{Q}_{E} \coloneqq i_2, \quad \text{and}\quad
    \mathcal{P}_{E} \coloneqq j_1. 
\]
These definitions are introduced in accordance with the construction of $\varphi^{*}$.

\begin{Example}
    Suppose that $g$ is a monomial in $R$ and takes the form $T_{A_1}\cdots T_{A_k}$.
    \begin{enumerate}[a]
        \item By the definition of $\varphi^*$, the multisets $\{\calX_{A_1},\dots,\calX_{A_k}\}$, $\{\calY_{A_1},\dots,\calY_{A_k}\}$, $\{\calQ_{A_1},\dots,\calQ_{A_k}\}$, and $\{\calP_{A_1},\dots,\calP_{A_k}\}$ are fiber-invariant. The degree $k$ is also fiber-invariant.
        \item Suppose that $\std(g)$ takes the form $T_{A_1'}\cdots T_{A_k'}$. We assert that \emph{the multiset $\{\XY_{A_1'},\dots,\XY_{A_k'}\}$ is fiber-invariant}. To demonstrate this, after reordering elements within the multisets, we may assume that the multiset $\{\calX_{A_1},\dots,\calX_{A_k}\}$ corresponds to $\{{i_1}\le \cdots \le {i_k}\}$, while the multiset $\{\calY_{A_1},\dots,\calY_{A_k}\}$ corresponds to $\{{j_1}\ge \cdots\ge {j_k}\}$. It is evident that the cells in the multiset $\{({i_1},{j_1}),\dots,({i_k},{j_k})\}$ are mutually NE-SW-compatible. As established in \Cref{rmk:GB_reduction}, the cells in the multiset $\{\XY_{A_1'},\dots,\XY_{A_k'}\}$ are also mutually NE-SW-compatible. Since both $\{\calX_{A_1},\dots,\calX_{A_k}\}$ and $\{\calY_{A_1},\dots,\calY_{A_k}\}$ are fiber-invariant multisets, the multiset $\{({i_1},{j_1}),\dots,({i_k},{j_k})\}$ necessarily coincides with $\{\XY_{A_1'},\dots,\XY_{A_k'}\}$. This observation substantiates the fiber-invariance property of $\{\XY_{A_1'},\dots,\XY_{A_k'}\}$. By analogous reasoning, the multiset $\{\QP_{A_1'},\dots,\QP_{A_k'}\}$ is also fiber-invariant.
    \end{enumerate} 
\end{Example}

We are now prepared to present the main result of this section, which establishes a quadratic Gr\"obner basis for the $\KK$-algebra $\KK[\LT(H)]$.

\begin{Theorem}
    \label{thm:varphi_star_GB}
    The set $\calF$ forms a quadratic Gr\"obner basis for $\ker(\varphi^*)$ under suitable monomial order on $R$. 
\end{Theorem}
\begin{proof}
    Let $J$ be the ideal in $R$, which is generated by $\calF$.  It is obvious that $J$ is a subideal of $\ker(\varphi^*)$. Furthermore, by \Cref{rmk:StandardMonomial}, these binomials form a Gr\"obner basis for $J$ under a suitable monomial order on $R$. It remains to show that $J$ coincides with $\ker(\varphi^*)$.

    Given that the defining ideal $\ker(\varphi^*)$ is binomial, it is generated by the binomials of the form $g-\widetilde{g}$, where $g$ is a monomial in $R$ and $\widetilde{g}$ is a distinct monomial in the fiber of $g$.
    It should be noted that every monomial $g$ in $R$ can be reduced to the standard monomial $\std(g)$
    through the reduction process by $\calF$. 
    When $\std(g)$ exhibits fiber-invariance, $\ker(\varphi^*)$ is generated by binomials of the form $g-\std(g)$, all of which belong to the subideal $J$. Consequently, the defining ideal $\ker(\varphi^*)$ coincides precisely with $J$. Based on this observation, it suffices to show that $\std(g)$ is fiber-invariant.

    We prove this by contradiction. Suppose that $g$ and $\widetilde{g}$ are two monomials of degree $k$ in the same fiber with $\std(g)\ne \std(\widetilde{g})$. Let $k$ be the smallest integer for which this occurs among the monomials in $R$. It is obvious that $k\ge 2$. Furthermore, without loss of generality, we can assume that $g=\std(g)$ and $\widetilde{g}=\std(\widetilde{g})$.

    Suppose that $g=\std(g)$ takes the form $T_{A_1}\cdots T_{A_k}$. The cells in $\{\XY_{A_1},\dots,\XY_{A_k}\}$ are mutually NE-SW-compatible and belong to $\Gamma$. Without loss of generality, for $1\le i<j\le k$, we may assume that $\XY_{A_i}$ lies weakly northeast of $\XY_{A_j}$. Furthermore, if $\XY_{A_i}= \XY_{A_j}$, we additionally assume that $\QP_{A_i}$ lies weakly northeast of $\QP_{A_j}$. Note that the cells in $\{\QP_{A_1},\dots,\QP_{A_k}\}$ belong to $\Gamma\cup \{(+\infty,-\infty)\}$. Whence, for $1\le i<j\le k$, if both $\QP_{A_i}$ and $\QP_{A_j}$ belong to $\Gamma$, we claim that $\QP_{A_i}$ lies weakly northeast of $\QP_{A_j}$. To verify this claim, it suffices to assume that $\XY_{A_i} \ne \XY_{A_j}$. Since $g=\std(g)$ is standard with respect to $J$, no nontrivial reduction of type (I\!I) can be applied further. Therefore, the claim holds.

    Denote the set $\{i\mid \QP_{A_i}=(+\infty,-\infty)\}=\{v_1<\cdots <v_t\}$ by $\calV=\calV(g)$ and the set $[k]\setminus\calV=\{u_1<\cdots<u_s\}$ by $\calU=\calU(g)$. The two cardinalities $s=\#\calU$ and $t=\#\calV$ are fiber-invariant, since $t$ is also the cardinality of $\{i\mid \mathcal{Q}_{A_i}=+\infty\}$ and $s=k-t$. Let us review the requirements that $\calU$ and $\calV$ must satisfy:
    \begin{enumerate}
        \item [\textup{(R$_1$)}] For all $i,j$ with $i<j$, the cell $\XY_{A_i}$ lies weakly northeast of the cell $\XY_{A_j}$ by the discussion above.
        \item [\textup{(R$_2$)}] For any $i\in \calV$ and $j\in \calU$ with $i<j$, the ordered pair $(\XY_{A_i},\QP_{A_j})$ is not legitimate with respect to $\Gamma$ in the sense of \Cref{rmk:NE_SW}. This is because no nontrivial reduction of type (I\!I\!I$_1$) can be applied to $\std(g)$.
        \item [\textup{(R$_3$)}] For all $i,j\in \calU$ with $i<j$, the cell $\QP_{A_i}$ lies weakly northeast of the cell $\QP_{A_j}$ by the discussion above.
        \item [\textup{(R$_4$)}] For every $i\in\calU$, the cell $\QP_{A_i}$ lies strictly southwest of the cell $\XY_{A_i}$ by the discussion in \Cref{rmk:diagram_and_variables} \ref{rmk:diagram_and_variables_c}.
    \end{enumerate}
    Since the multisets $\{\XY_{A_1},\dots,\XY_{A_k}\}$ and $\{\QP_{A_1},\dots,\QP_{A_k}\}$ are fiber-invariant, the set $\calU$ (and equivalently $\calV$) uniquely determines the monomial $\std(g)$. It remains to show that $\calV$ is fiber-invariant. 

    Furthermore, we can assume that the requirements (R$_1$)-(R$_4$) are also satisfied by the corresponding objects obtained from $\widetilde{g}=\std(\widetilde{g})=T_{\widetilde{A}_1}\cdots T_{\widetilde{A}_k}$. In particular, let $\widetilde{\calU} \coloneqq \calU(\widetilde{g})
    $ and $\widetilde{\calV}\coloneqq\calV(\widetilde{g})
    $ be the corresponding sets obtained from $\widetilde{g}$. We have pointed out that $\#\calU=\#\widetilde{\calU}$ and $\#\calV=\#\widetilde{\calV}$ by the fiber-invariance. Since $g=\std(g)\ne \widetilde{g}=\std(\widetilde{g})$, it follows that $\calV\ne \widetilde{\calV}$ and $\calU\ne \widetilde{\calU}$. Note that factors of standard monomials are still standard. By the minimality assumption of $k$, the subsets $\calV$ and $\widetilde{\calV}$ are therefore disjoint. Let $\ell\coloneqq \min(\calV)$, and suppose without loss of generality that $\ell<\widetilde{\ell}\coloneqq \min (\widetilde{\calV})$.

    Since the multisets $\{\XY_{A_1},\dots,\XY_{A_k}\}$ is fiber-invariant, $\XY_{A_i}=\XY_{\widetilde{A}_i}$ for all $i$ by requirement (R$_1$). Since $\{\QP_{A_1},\dots,\QP_{A_k}\}$ is fiber-invariant, it follows from requirement (R$_3$) that $\QP_{A_i}=\QP_{\widetilde{A}_i}$ for $i<\ell$. Consequently, $T_{A_i}=T_{\widetilde{A}_i}$ for $i<\ell$. By the minimality assumption on $k$, we must therefore have $\ell=1$.

    Since $\calV$ and $\widetilde{\calV}$ are disjoint, it is evident that $\widetilde{\ell}\in \calU$. Let $\ell'\coloneqq\min(\calU)$. It is clear that $\ell=1<\ell'\le \widetilde{\ell}$. By requirement (R$_2$), the ordered pair $(\XY_{A_{1}},\QP_{A_{\ell'}})$ is not legitimate with respect to $\Gamma$. On the other hand, since $1=\min(\widetilde{\calU})$ and $\ell'=\min(\calU)$, by requirement (R$_3$), we must have $\QP_{\widetilde{A}_1}=\QP_{A_{\ell'}}$.  In particular, the ordered pair $(\XY_{\widetilde{A}_1},\QP_{\widetilde{A}_1})=(\XY_{A_1},\QP_{A_{\ell'}})$ is legitimate with respect to $\Gamma$. 
    This yields the desired contradiction.
\end{proof}

\section{\textsc{Sagbi} basis}\label{sec:Sagbi}
To establish our main theorem, it remains to show that the Gr\"ober basis presented in \Cref{thm:varphi_star_GB} can be ``lifted'' to form a Gr\"obner basis for the kernel of $\varphi$, as illustrated in \cite[Corollary 2.1, 2.2]{Sagbi}.

In the special case where $\Gamma=\overline{\Gamma}$, this has been thoroughly addressed by Higashitani \cite{arXiv:2411.07812}. Based on the relative comparison of the subscripts, he identified $13$ types of binomials in the kernel of $\varphi^*$ to verify the lifting.

In the more general case, $\Gamma$ is a subdiagram of $\overline{\Gamma}$. It is natural to hope that a canonical restriction of his lifting might work in this case. However, this is not true in general. 
When starting with a quadratic binomial whose set of cells is contained in $\Gamma$, Higashitani's lifting might still involve cells in $\overline{\Gamma}\setminus \Gamma$. Furthermore, we have the extra binomial, corresponding to the reduction in \eqref{eqn:reduction_42}, that awaits lifting.

\subsection{The $13$ cases and the lifting}
\label{subsection:lifting}
In what follows, we will scrutinize the $13$ cases of lifted polynomials examined by Higashitani. We will check whether Higashitani's lifting method still works in each cases. If not, we will provide remedies.
To facilitate the argument, we introduce some concepts.

\begin{Definition}
    Assume that $\calZ$ is a nonempty subset of $\Gamma$. Let $\cl(\calZ)$ be the minimal subset of $\overline{\Gamma}$ such that the following is satisfied:
    \begin{enumerate}[a]
        \item the set $\calZ$ is a subset of $\cl(\calZ)$;
        \item for each NW-SE-compatible pair $(A,B)$ with $A,B\in \cl(\calZ)$, the cells in the rectangular region determined by $A$ and $B$ also belong to $\cl(\calZ)$.
    \end{enumerate} 
    It follows from the discussion in \Cref{rmk:diagram_and_variables}\ref{rmk:diagram_and_variables_a} that $\cl(\calZ)$ is a subset of $\Gamma$. We will refer to $\cl(\calZ)$ as the \emph{closure} of $\calZ$ in $\Gamma$.
\end{Definition}

\begin{Example}
    Let us consider two quadratic monomials $T_{1,1;3,2}T_{2,3;4,4}$ and $T_{2,6;3,9}T_{1,7;4,8}$, which correspond to subsequent polynomials $F_5$ and $F_{12}$ respectively. The cell set of $T_{1,1;3,2}T_{2,3;4,4}$ is denoted by $\times$ in \Cref{fig:cloure}, while its closure is represented by dotted markers. On the other hand, the cell set of $T_{2,6;3,9}T_{1,7;4,8}$ is indicated by $\triangle$, with its closure illustrated by gray shading.
    \begin{figure}[tbhp]
        \begin{tikzpicture}[x=2\unitlength,y=2\unitlength,inner sep=0pt]                   
            \fill [pattern = crosshatch dots, pattern color=gray!50] (0,0) rectangle (40,40);
            \fill[fill=gray!20] (50,0) rectangle (80,30);
            \fill[fill=gray!20] (60,10) rectangle (90,40);
            \foreach \x in {0,1,2,3,4} \draw (0,\x*10) -- (90,\x*10); 
            \foreach \x in {1,2,3,4} 
            \draw [color=gray] (-5,45-\x*10) node {$\x$};
            \foreach \x in {0,...,9} \draw (\x*10,0) -- (\x*10,40); 
            \foreach \x in {1,...,9} \draw [color=gray] (\x*10-5,45) node {$\x$}; 
            \draw (5,15) node {$\times$};
            \draw (5,35) node {$\times$};
            \draw (15,35) node {$\times$};
            \draw (15,15) node {$\times$};
            \draw (25,5) node {$\times$};
            \draw (25,25) node {$\times$};
            \draw (35,25) node {$\times$};
            \draw (35,5) node {$\times$};

            \draw (55,15) node {$\triangle$};
            \draw (55,25) node {$\triangle$};
            \draw (65,35) node {$\triangle$};
            \draw (65,5) node {$\triangle$};
            \draw (75,5) node {$\triangle$};
            \draw (75,35) node {$\triangle$};
            \draw (85,25) node {$\triangle$};
            \draw (85,15) node {$\triangle$};
        \end{tikzpicture}
        \caption{Closures of two sets of cells}
        \label{fig:cloure}
    \end{figure}
\end{Example}

Throughout this discussion, we consistently assume the following index ranges: $1 \le i,i',e,e' \le a$ and $1 \le j,j',f,f' \le b$, with the constraints $i<i'$, $e<e'$, $j'<j$ and $f'<f$. 
Moreover, in each specified polynomial $F_*$, the first monomial has a set of cells contained in $\Gamma$.

\begin{enumerate}[i]
    \item $F_1\coloneqq \uline{T_{i,j}T_{e,f}-T_{i,f}T_{e,j}}-T_{i,j;e,f}$ with $i<e$ and $j<f$. 

        The marked binomial is of type \eqref{eqn:reduction_22}. It is trivially lifted to the polynomial $F_1$ in $R$. 
    \item $F_2\coloneqq\uline{T_{i,j}T_{e,f';e',f}-T_{e,f}T_{i,f';e',j}}+T_{i,f'}T_{e,f;e',j}+T_{e',j}T_{i,f';e,f}+T_{e',f'}T_{i,f;e,j}$ with $i<e$ and $f<j$. 

        If $T_{i,f';e',j}\in R$, the marked binomial is of type \eqref{eqn:reduction_24}. The rectangular region determined by $(i,f')$ and $(e',j)$ is contained in $\Gamma$. Since the cell set of $F_2$ is contained in this region, $F_2$ is an element of $R$. Consequently, the stated lifting still holds.

        If, instead, $T_{i,f';e',j}\notin R$, the monomial $T_{i,j}T_{e,f';e',f}$ is already in standard form under our reduction. There is nothing to concern in this case.

    \item 
        $F_3\coloneqq \uline{T_{i,j}T_{e,f';e',f}-T_{i,f}T_{e,f';e',j}}+T_{i,f'}T_{e,f;e',j}$ with $e \le i$ and $f<j$. 

        If $i\ge e'$, the marked binomial is of type \eqref{eqn:reduction_24}. The rectangular region determined by $(e,f')$ and $(i,j)$ is contained in $\Gamma$. Since the cell set of $F_3$ is contained in this region, $F_3$ belongs to $R$. Consequently, the stated lifting still holds.

        If $i<e'$ and $T_{e,f';e',j}\in R$, the marked binomial is of type \eqref{eqn:reduction_24}. The rectangular region determined by $(e,f')$ and $(e',j)$ is contained in $\Gamma$. Since the cell set of $F_3$ is contained in this region, $F_3$ belongs to $R$. Consequently, the stated lifting still holds.

        If $i<e'$ and $T_{e,f';e',j}\notin R$, we will use the reduction \eqref{eqn:reduction_42}. Whence, we turn to 
        \begin{align*}
            F_3'\coloneqq \uline{T_{i,j}T_{e,f';e',f}-T_{e,j}T_{i,f';e',f}}-T_{e',f}T_{e,f';i,j}+T_{e',f'}T_{e,f;i,j}\in \ker(\varphi). 
        \end{align*}
        Notice that the marked binomial in 
        \(F_3'\)
        belongs to the Gr\"obner basis $\calF$. Since
        \begin{align*}
            \ini_{>_{\lex}}(\varphi(T_{i,j}T_{e,f';e',f}-T_{e,j}T_{i,f';e',f})) = x_e x_{e'} p_{f'} q_i y_f y_j ,  \\
            \ini_{>_{\lex}}(\varphi(T_{e',f}T_{e,f';i,j})) = x_e x_{e'} p_{f'} q_i y_f y_j , \\
            \ini_{>_{\lex}}(\varphi(T_{e',f'}T_{e,f;i,j}))= x_e x_{e'} p_f q_i y_{f'} y_j, 
        \end{align*}
        the polynomial in 
        $F_3'$ satisfies the lifting requirement and provides the remedy in this case.

    \item $F_4\coloneqq \uline{T_{i,j}T_{e,f';e',f}-T_{e,j}T_{i,f';e',f}}+T_{e',j}T_{i,f';e,f}$ with $i < e$ and $j \le f$. 

        If $j\le f'$, the marked binomial is of type \eqref{eqn:reduction_24}. The rectangular region determined by $(i,j)$ and $(e',f)$ is contained in $\Gamma$. Since the cell set of $F_4$ is contained in this region, $F_4$ belongs to $R$. Consequently, the stated lifting still holds.

        If instead $j> f'$ and $T_{i,f';e',f}\in R$, the marked binomial is of type \eqref{eqn:reduction_24}. The rectangular region determined by $(i,f')$ and $(e',f)$ is contained in $\Gamma$. Since the cell set of $F_4$ is contained in this region, $F_4$ belongs to $R$. Consequently, the stated lifting still holds.

        If $j> f'$ and $T_{i,f';e',f}\notin R$, we will use the reduction \eqref{eqn:reduction_42}. Whence, we turn to 
        \begin{align*}
            F_4'\coloneqq \uline{T_{i,j}T_{e,f';e',f}-T_{i,f}T_{e,f';e',j}}-T_{e,f'}T_{i,j;e',f}+T_{e',f'}T_{i,j;e,f}\in \ker(\varphi). 
        \end{align*}
        Notice that the marked binomial in
        $F_4'$ belongs to the Gr\"obner basis $\calF$. Since
        \begin{align*}
            \ini_{>_{\lex}}(\varphi(T_{i,j}T_{e,f';e',f}-T_{i,f}T_{e,f';e',j}) = x_i x_{e} p_{j} q_{e'} y_{f'} y_f ,  \\
            \ini_{>_{\lex}}(\varphi(T_{e,f'}T_{i,j;e',f})) = x_i x_{e} p_{j} q_{e'} y_{f'} y_f , \\
            \ini_{>_{\lex}}(\varphi(T_{e',f'}T_{i,j;e,f}))= x_i x_{e'} p_j q_e y_{f'} y_f, 
        \end{align*}
        the polynomial 
        $F_4'$ satisfies the lifting requirement and provides the remedy in this case.

    \item $F_5\coloneqq \uline{T_{i,j';i',j}T_{e,f';e',f}-T_{i,f';i',f}T_{e,j';e',j}}$ with $i \le e$, $i' \le e'$, $j' \le f'$, and $j \le f$. 

        Since this polynomial is a binomial and belongs to $\ker(\varphi)$, it is already lifted.

    \item $F_6\coloneqq \uline{T_{i,j';i',j}T_{e,f';e',f}-T_{e,f';i',f}T_{i,j';e',j}}+T_{e,f';i,f}T_{i',j';e',j}$ with $i > e$, $i' < e'$, $j' \le f'$, and $j \le f$. 

        Whether $j\ge f'$ or not, it is straightforward to verify that the cell set of $F_6$ is contained in the closure $\cl(\Cell(T_{i,j';i',j}T_{e,f';e',f}))$. Consequently, the marked binomial is of type \eqref{eqn:reduction_44}, and $F_6$ belongs to $R$. Therefore, the stated lifting still holds.

    \item $F_7\coloneqq \uline{T_{i,j';i',j}T_{e,f';e',f}-T_{i,f';e',f}T_{e,j';i',j}}+T_{e',j';i',j}T_{i,f';e,f}$ with $i < e$, $i' > e'$, $j' \le f'$, and $j \le f$. 

        Whether $j\ge f'$ or not, it is straightforward to verify that the cell set of $F_7$ is contained in the closure $\cl(\Cell(T_{i,j';i',j}T_{e,f';e',f}))$. Consequently, the marked binomial is of type \eqref{eqn:reduction_44}, and $F_7$ belongs to $R$. Therefore, the stated lifting still holds.

    \item $F_8\coloneqq \uline{T_{i,j';i',j}T_{e,f';e',f}-T_{i,j';i',f}T_{e,f';e',j}}+T_{i,j;i',f}T_{e,f';e',j'}$ with $i \le e$, $i' \le e'$, $j' > f'$ and $j < f$. 

        Whether $i'\ge e$ or not, it is straightforward to verify that the cell set of $F_8$ is contained in the closure $\cl(\Cell(T_{i,j';i',j}T_{e,f';e',f}))$. Consequently, the marked binomial is of type \eqref{eqn:reduction_44}, and $F_8$ belongs to $R$. Therefore, the stated lifting still holds.

    \item $F_9\coloneqq \uline{T_{i,j';i',j}T_{e,f';e',f}-T_{i,f';i',j}T_{e,j';e',f}}+T_{i,j';i',f'}T_{e,f;e',j}$ with $i \le e$, $i' \le e'$, $j' < f'$, and $j>f$. 

        Since $T_{e,f;e',j}$ might not be available in $R$, $F_9$ is not the optimal lifting in this case. We can instead use
        \begin{equation*}
            F_9'\coloneqq\uline{T_{i,j';i',j}T_{e,f';e',f}-T_{i,f';i',j}T_{e,j';e',f}}+T_{i,f;i',j}T_{e,j';e',f'} \in \ker(\varphi), 
        \end{equation*}
        which has the same underlined binomial.
        Whether $i'\ge e$ or not, it is straightforward to verify that the cell set of $F_9'$ is contained in the closure $\cl(\Cell(T_{i,j';i',j}T_{e,f';e',f}))$. Consequently, the marked binomial is of type \eqref{eqn:reduction_44}, and $F_9'$ belongs to $R$. Therefore, the stated lifting still holds.
        Furthermore, since this is a three-terms polynomial, it is trivially that
        \[
            \ini_{>_{\lex}}(\varphi(T_{i,j';i',j}T_{e,f';e',f}-T_{i,f';i',j}T_{e,j';e',f}))=- \ini_{>_{\lex}}(\varphi(T_{i,f;i',j}T_{e,j';e',f'})).
        \]
        Thus, 
        the polynomial $F_9'$
        satisfies the lifting requirement and provides the remedy in this case.

    \item
        $F_{10}\coloneqq \uline{T_{i,j';i',j}T_{e,f';e',f}-T_{e,j';i',f}T_{i,f';e',j}}+T_{e,j';i,f}T_{i',f';e',j}+T_{i,j;e',f}T_{e,f';i',j'}-T_{i',j;e',f}T_{e,f';i,j'}$ with $i > e$, $i' < e'$, $j' > f'$, and $j < f$. 

        It is straightforward to verify that the cell set of $F_{10}$ is contained in the closure $\cl(\Cell(T_{i,j';i',j}T_{e,f';e',f}))$. Consequently, the marked binomial is of type \eqref{eqn:reduction_44}, and $F_{10}$ belongs to $R$. Therefore, the stated lifting still holds.

    \item $F_{11}\coloneqq \uline{T_{i,j';i',j}T_{e,f';e',f}-T_{e,j';i',f}T_{i,f';e',j}}+T_{i,f';e,j}T_{e',j';i',f}+T_{e,f;i',j}T_{i,j';e',f'}-T_{e',f;i',j}T_{i,j';e,f'}$ with $i < e$, $i' > e'$, $j' < f'$, and $j>f$. 

        It is straightforward to verify that the cell set of $F_{11}$ is contained in the closure $\cl(\Cell(T_{i,j';i',j}T_{e,f';e',f}))$. Consequently, the marked binomial is of type \eqref{eqn:reduction_44}, and $F_{11}$ belongs to $R$. Therefore, the stated lifting still holds.

    \item \label{item_sagbi_l}
        $F_{12}\coloneqq \uline{T_{i,j';i',j}T_{e,f';e',f}-T_{e,f';i',j}T_{i,j';e',f}}+T_{e,f';i,j}T_{i',j';e',f}+T_{i,f;e',j}T_{e,j';i',f'}-T_{i',f;e',j}T_{e,j';i,f'} $ with $i > e$, $i' < e'$, $j' < f'$, and $j > f$.

        Since $T_{i,f;e',j}$ might not be available in $R$, $F_{12}$ is not the optimal lifting in this case. We can instead use
        \begin{align*}
            F_{12}'\coloneqq &\uline{T_{i,j';i',j}T_{e,f';e',f}-T_{e,f';i',j}T_{i,j';e',f}}+T_{e,f;i',j}T_{i,j';e',f'} \notag\\
            &\qquad\qquad\qquad -T_{e,f;i,j}T_{i',j';e',f'}+T_{e,f';i,j}T_{i',j';e',f}\in \ker(\varphi), 
        \end{align*}
        which has the same underlined binomial.
        Since
        \begin{align*}
            \ini_{>_{\lex}}(\varphi(T_{i,j';i',j}T_{e,f';e',f}-T_{e,f';i',j}T_{i,j';e',f})) &= -x_e x_i p_{j'} p_f q_{i'} q_{e'} y_{f'} y_j, \\
            \ini_{>_{\lex}}(\varphi(T_{e,f;i'j}T_{i,j';e',f'}))&=x_e x_i p_{j'} p_f q_{i'} q_{e'} y_{f'} y_j, \\
            \ini_{>_{\lex}}(\varphi(T_{e,f;i,j}T_{i',j';e',f'}))&=x_e x_{i'} p_{j'} p_f q_i q_{e'} y_{f'} y_j, \\
            \ini_{>_{\lex}}(\varphi(T_{e,f';i,j}T_{i',j';e',f}))&=x_e x_{i'} p_{j'} p_{f'} q_i q_{e'} y_f y_j,
        \end{align*}
        the polynomial $F_{12}'$
        satisfies the lifting requirement and provides the remedy in this case.
    \item $F_{13}\coloneqq \uline{T_{i,j';i',j}T_{e,f';e',f}-T_{e,f';i',j}T_{i,j';e',f}}+T_{i,j';e,f}T_{e',f';i',j}+T_{e,j;i',f}T_{i,f';e',j'}-T_{e',j;i',f}T_{i,f';e,j'}$ with $i < e$, $i' > e'$, $j' > f'$, and $j < f$.

        This case is identical to the case in \ref{item_sagbi_l} by symmetry with respect to the anti-diagonal line.
\end{enumerate} 

\begin{Remark}
    A thorough examination reveals that when we initiate with a quadratic monomial $g$ in $R$, the discussion above appears incomplete. The omitted cases arise from two primary considerations: First, certain cases are excluded due to symmetry, as their inclusion would be redundant. Second, the monomial $g$ is already in its standard form, consequently precluding the possibility of any binomial in the Gr\"obner basis $\calF$ having $g$ as its leading monomial.
\end{Remark}

\subsection{Proof of the main theorem}
After verifying the proper lifting of binomials generators of the defining ideal of $\KK[\LT(H)]$, we proceed to present the standard proof of our main theorem. Let $\calJ$ be the set of properly lifted polynomials $F_i$ and $F_i'$
as defined in the previous subsection.

\begin{Theorem}
    \begin{enumerate}[a]
        \item \label{thm:main_a}
            The set $H$ forms a \textsc{Sagbi} basis for the binomial edge ring $\KK[G(J_\calG)]$ with respect to the lexicographic order $>_{\lex}$ on $S$. In particular,  $\ini_{>_\lex}(\KK[G(J_\calG)])=\KK[\LT(H)]$. 
        \item \label{thm:main_b}
            The set $\calJ$ is a quadratic G\"{o}bner basis of the defining ideal of binomial edge ring $\KK[G(J_\calG)]$. 
        \item The binomial edge ring $\KK[G(J_\calG)]$ is a Koszul Cohen-Macaulay normal domain. Moreover, when $\KK$ is of characteristic $0$, $\KK[G(J_\calG)]$  has rational singularities; when $\KK$ is of positive characteristic, $\KK[G(J_\calG)]$ is $F$-rational.
    \end{enumerate}
\end{Theorem}

\begin{proof}
    The results \ref{thm:main_a} and \ref{thm:main_b} are direct consequences of \cite[Proposition 1.1 and Corollaries 2.1, 2.2]{Sagbi}, \Cref{thm:varphi_star_GB}, and the lifting verification presented in Subsection \ref{subsection:lifting}.
    Since $\calJ$ is a quadratic Gr\"{o}bner basis for the defining ideal of $\KK[G(J_\calG)]$, it follows from \cite[Theorem 6.7]{EHgbBook} that $\KK[G(J_\calG)]$ is Koszul. Moreover, its initial algebra $\KK[\LT(H)]$ is a semigroup ring whose defining ideal has a squarefree initial ideal. According to \cite[Proposition 13.15]{Sturmfels}, $\KK[\LT(H)]$ is normal, and consequently Cohen--Macaulay, by \cite[Theorem 1]{Hochster}. The remaining conclusions are derived from \cite[Corollary 2.3]{Sagbi}.
    \qedhere 
\end{proof}

\section{Application: the Krull dimension}

We conclude this work by providing a combinatorial formula for the Krull dimension of the binomial edge ring $\KK[G(J_{\calG})]$. More precisely, the formula is given through the enumeration of specific cells within the corresponding Ferrers diagram. We begin by defining these relevant cells.

\begin{Definition}
    Let $(i,j)$ be a cell in $\Gamma$. If its northwestern neighbor $(i-1,j-1)$ does not lie in $\Gamma$, we call $(i,j)$ a \emph{northwestern perimeter cell} of $\Gamma$. If the northwestern neighbor $(i-1,j-1)$ lie in $\Gamma$, but the southeastern neighbor $(i+1,j+1)$ does not, we call $(i,j)$ a \emph{southeastern perimeter cell} of $\Gamma$. The set of northwestern perimeter cells is denoted by $\bdP_{NW}=\bdP_{NW}(\Gamma)$, and the set of southeastern perimeter cells is denoted by $\bdP_{SE}=\bdP_{SE}(\Gamma)$. Finally, let $\bdP=\bdP_{NW}\sqcup \bdP_{SE}$.
\end{Definition}

The following example provides a visual illustration of the definitions above.

\begin{Example} 
    \Cref{fig:NWP} shows the skew Ferrers diagram 
    $\Gamma=\bdlambda/\bdmu$ 
    for $\bdlambda=(9,8,8,6,5,5,5,2,2)$ and $\bdmu=(5,4,4,3,2,2,2,0,0)$. Its northwestern perimeter cells are marked with $\triangle$, and its southeastern perimeter cells are marked with $\times$. 
    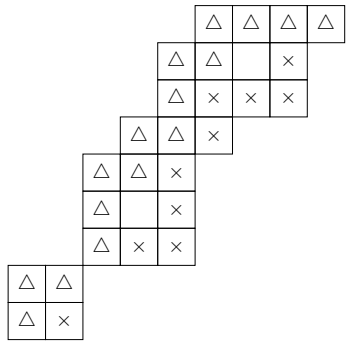
\begin{figure}[h]
        \scalebox{0.7}{
            \begin{tikzpicture}[x=2\unitlength,y=2\unitlength,inner sep=0pt]
                \draw (45,90) rectangle (55,80);
                \draw (35,80) rectangle (45,70);
                \draw (45,80) rectangle (55,70);
                \draw (55,90) rectangle (65,80);
                \draw (55,80) rectangle (65,70);
                \draw (65,80) rectangle (75,70);
                \draw (65,90) rectangle (75,80);
                \draw (75,90) rectangle (85,80);
                \draw (35,70) rectangle (45,60);
                \draw (45,70) rectangle (55,60);
                \draw (55,70) rectangle (65,60);
                \draw (65,70) rectangle (75,60);
                \draw (25,60) rectangle (35,50);
                \draw (35,60) rectangle (45,50);
                \draw (45,60) rectangle (55,50);
                \draw (15,50) rectangle (25,40);
                \draw (25,50) rectangle (35,40);
                \draw (35,50) rectangle (45,40);
                \draw (15,40) rectangle (25,30);
                \draw (25,40) rectangle (35,30);
                \draw (35,40) rectangle (45,30);
                \draw (15,30) rectangle (25,20);
                \draw (25,30) rectangle (35,20);
                \draw (35,30) rectangle (45,20);
                \draw (-5,20) rectangle (5,10);
                \draw (5,20) rectangle (15,10);
                \draw (-5,10) rectangle (5,0);
                \draw (5,10) rectangle (15,0);
                \draw (70,75) node {$\times$};
                \draw (60,65) node {$\times$};
                \draw (70,65) node {$\times$};
                \draw (40,45) node {$\times$};
                \draw (50,55) node {$\times$};
                \draw (50,65) node {$\times$};
                \draw (40,35) node {$\times$};
                \draw (40,25) node {$\times$};
                \draw (30,25) node {$\times$};
                \draw (10,5) node {$\times$};
                \draw (80,85) node {$\triangle$};
                \draw (20,25) node {$\triangle$};
                \draw (20,35) node {$\triangle$};
                \draw (20,45) node {$\triangle$};
                \draw (30,45) node {$\triangle$};
                \draw (40,55) node {$\triangle$};
                \draw (30,55) node {$\triangle$};
                \draw (40,65) node {$\triangle$};
                \draw (40,75) node {$\triangle$};
                \draw (50,75) node {$\triangle$};
                \draw (50,85) node {$\triangle$};
                \draw (60,85) node {$\triangle$};
                \draw (70,85) node {$\triangle$};
                \draw (0,15) node {$\triangle$};
                \draw (0,5) node {$\triangle$};
                \draw (10,15) node {$\triangle$};
        \end{tikzpicture}}
        \caption{Perimeter cells}
        \label{fig:NWP}
    \end{figure}
\end{Example}

The final theorem of this work is stated as follows.

\begin{Theorem}
    \label{thm:dim}
    The Krull dimension of $\KK[G(J_\calG)]$ is equal to the cardinality $\#\bdP$.
\end{Theorem}
\begin{proof}
    Since $\KK[\LT(H)]$ is the initial algebra of $\KK[G(J_\calG)]$,
    it follows from \cite[Theorem A.16]{BH} and \cite[Corollary 2.3]{Sagbi} that 
    \begin{align*}
        \dim(\KK[G(J_\calG)])&= \dim(\KK[\LT(H)])\\
        &=\dim(\KK[\{x_iy_j\mid T_{i,j}\in R\}\cup\{x_{i_1}p_{j_1}q_{i_2}y_{j_2}\mid T_{i_1,j_1;i_2,j_2}\in R\}])\\
        &=\trdeg_{\KK}(\KK(\{x_iy_j\mid T_{i,j}\in R\}\cup\{x_{i_1}p_{j_1}q_{i_2}y_{j_2}\mid T_{i_1,j_1;i_2,j_2}\in R\}))\\
        &=\trdeg_{\KK}(\KK(\{x_iy_j\mid T_{i,j}\in R\}\cup\{p_{j_1}q_{i_2}\mid T_{i_1,j_1;i_2,j_2}\in R\}))\\
        &=\trdeg_{\KK}(\KK(\{x_iy_j\mid T_{i,j}\in R\}))+\trdeg_{\KK}(\KK(\{p_{j_1}q_{i_2}\mid T_{i_1,j_1;i_2,j_2}\in R\})).
    \end{align*}
    Since $\#\bdP_{NW}+\#\bdP_{SE} =\#\bdP$, it remains to show that
    \begin{align}
        \trdeg_{\KK}(\KK(\{x_iy_j\mid T_{i,j}\in R\}))&= \#\bdP_{NW}
        \label{eqn:trdeg_NW}
        \intertext{and}
        \trdeg_{\KK}(\KK(\{p_{j_1}q_{i_2}\mid T_{i_1,j_1;i_2,j_2}\in R\}))&=
        \#\bdP_{SE}.
        \label{eqn:trdeg_SE}
    \end{align}
    The remaining proof is divided into two parts.
    \begin{enumerate}[a]
        \item Let us first prove the equality in \eqref{eqn:trdeg_NW}.
            If $(i,j),(i,j+1)\in \Gamma$, we write $(i,j)\sim (i,j+1)$. Likewise, if $(i,j),(i+1,j)\in \Gamma$, we also write $(i,j)\sim(i+1,j)$. Clearly, $\sim$ generates an equivalence on the cells of $\Gamma$. The induced equivalence classes will be called \emph{edge-connected components}. Suppose that $\Gamma_1,\dots,\Gamma_t$ are the edge-connected components of $\Gamma$.  The following are clear:
            \begin{enumerate}[i]
                \item each $\Gamma_k$ is still a skew Ferrers diagram;
                \item $\trdeg_{\KK}(\KK(\{x_iy_j\mid T_{i,j}\in R\}))=\sum_{k=1}^t\trdeg_{\KK}(\KK(\{x_iy_j\mid (i,j)\in \Gamma_k\}))$;
                \item $\bdP_{NW}(\Gamma)\cap \Gamma_k=\bdP_{NW}(\Gamma_k)$.
            \end{enumerate} 
            Therefore, to establish \eqref{eqn:trdeg_NW}, it suffices to assume that $k=1$ and $\Gamma$ is edge-connected.

            Without loss of generality, we can further assume that $\mu_a=0$.  Whence, the northeast extreme cell of $\Gamma$ is $(1,b)$ and the southwest extreme cell of $\Gamma$ is $(a,1)$. Meanwhile, the cells of $\bdP_{NW}$ form a path connecting $(a,1)$ with $(1,b)$ inside $\Gamma$. In particular, $
            \#\bdP_{NW}
            =a+b-1$. We will order the cells on this path such that $u_1=(1,b)$ and $u_{a+b-1}=(a,1)$. For each $k=1,2,\dots, a-b-2$, if $u_k=(i,j)$, then $u_{k+1}=(i,j-1)$ or $(i+1,j)$. Whence $\varphi^{*}(T_{u_{k+1}})/\varphi^{*}(T_{u_k})=y_{j-1}/y_j$ or $x_{i+1}/x_i$. Consequently, 
            \begin{align*}
                \KK(x_iy_j\mid (i,j)\in \bdP_{NW})&=\KK\left(x_{1}y_b,\frac{x_2}{x_1},\frac{x_3}{x_2},\dots,\frac{x_a}{x_{a-1}},\frac{y_1}{y_2},\dots,\frac{y_{b-2}}{y_{b-1}},\frac{y_{b-1}}{y_b}\right)\\
                &=\KK\left(x_1y_b,\frac{x_2}{x_1},\frac{x_3}{x_1},\dots,\frac{x_a}{x_1},\frac{y_1}{y_b},\dots,\frac{y_{b-2}}{y_{b}}, \frac{y_{b-1}}{y_b}\right).
            \end{align*}
            Denote this field by $K$. It follows that $\trdeg_{\KK}(K)=a+b-1=
            \#\bdP_{NW}
            $.

            For every $(i+1,j+1)\in \Gamma\setminus \bdP_{NW}$, we have $(i,j),(i+1,j),(i,j+1)\in \Gamma$. Furthermore, from $\varphi^*(T_{i,j}),\varphi^*(T_{i+1,j}),\varphi^*(T_{i,j+1})\in K$, we see that 
            \[
                \varphi^*(T_{i+1,j+1})=\varphi^*(T_{i+1,j})\varphi^*(T_{i,j+1})/\varphi^*(T_{i,j})\in K.
            \]
            Therefore, since $\varphi^\ast(T_{i,j})\in K$ for every $(i,j)\in \bdP_{NW}$, it follows by easy induction that $\varphi^\ast(T_{i,j})\in K$ for every $(i,j)\in \Gamma$. In particular, $K =\KK(\{x_iy_j\mid T_{i,j}\in R\})$. In summary, we have established the equality in \eqref{eqn:trdeg_NW}.

        \item Secondly, we prove the equality in \eqref{eqn:trdeg_SE}. Note that if $T_{i_1,j_1;i_2,j_2}\in R$, then all cells in the rectangular region determined by $(i_1,j_1)$ and $(i_2,j_2)$ are contained in $\Gamma$. In particular, the rectangular region determined by $(i_2-1,j_1;i_2,j_1+1)$ is contained in $\Gamma$. In other words, $T_{i_2-1,j_1;i_2,j_1+1}\in R$. Observe that 
            \[
                \QP_{T_{i_1,j_1;i_2,j_2}}= q_{i_2}p_{j_1} =\QP_{T_{i_2-1,j_1;i_2,j_1+1}}.
            \]
            Inspired by this observation, we refer to the set $\{(i-1,j-1),(i-1,j),(i,j-1),(i,j)\}$ as an \emph{adjacent square region}, denoted by $\sA_{i,j}$.
            Therefore, 
            \[
                \KK(\{p_{j_1}q_{i_2}\mid T_{i_1,j_1;i_2,j_2}\in R\}) = \KK(\{p_{j-1}q_i\mid \sA_{i,j}\subseteq \Gamma\})\isom \KK(\{p_{j}q_i\mid \sA_{i,j}\subseteq \Gamma\}).
            \]
            Notice that $\sA_{i,j}\subseteq \Gamma$ if and only if $(i,j)\in \Gamma'\coloneqq \Gamma\setminus \bdP_{NW}$. 
            Furthermore, we introduce the rotated diagram $\Gamma''\coloneqq\{(a+1-i,b+1-j)\mid (i,j)\in \Gamma'\}$.
            It follows from \eqref{eqn:trdeg_NW} that
            \begin{align*}
                \trdeg_{\KK}(\KK(\{p_{j}q_i\mid \sA_{i,j}\subseteq \Gamma\}))&=
                \trdeg_{\KK}(\KK(\{p_{j}q_i\mid (i,j)\subseteq \Gamma'\}))\\
                &=
                \trdeg_{\KK}(\KK(\{p_{b+1-j}q_{a+1-i}\mid (i,j)\subseteq \Gamma'\}))\\
                &\xlongequal[\text{isomorphism}]{\text{under obvious}}\trdeg_{\KK}(\KK(\{y_{j}x_{i}\mid (i,j)\subseteq \Gamma''\}))\\
                &=\#\bdP_{NW}(\Gamma'').
            \end{align*}
            It remains to notice that $\bdP_{NW}(\Gamma'')$ is precisely the $\bdP_{SE}(\Gamma)$ under the rotation. In particular, 
            $\#\bdP_{NW}(\Gamma'')=\#\bdP_{SE}(\Gamma)$, and we have thus established the equality in \eqref{eqn:trdeg_SE}. This completes the proof. 
            \qedhere
    \end{enumerate} 
\end{proof}

\begin{Example}
    As an illustration, let us consider the diagram depicted in \Cref{fig:NWP} as $\Gamma$. It has two edge-connected components. Furthermore, the corresponding $\Gamma'$ introduced in the proof of \Cref{thm:dim} is drawn in the left subfigure of \Cref{fig:Gamma23}, with the cells in $\bdP_{SE}(\Gamma)$ still marked by $\times$. Around $\Gamma'$, a dotted frame representing the boundary of the diagram $\overline{\Gamma}$ is drawn for positional reference.  Correspondingly, the rotated diagram $\Gamma''$ is drawn in the right subfigure of \Cref{fig:Gamma23}, whose northwestern perimeter cells are marked by $\triangle$.
\end{Example}

\begin{figure}[htbp]
    \scalebox{0.7}{
        \begin{tikzpicture}[x=2\unitlength,y=2\unitlength,inner sep=0pt]  
            \draw (70,75) node {$\times$};
            \draw (60,65) node {$\times$};
            \draw (70,65) node {$\times$};
            \draw (40,45) node {$\times$};
            \draw (50,55) node {$\times$};
            \draw (50,65) node {$\times$};
            \draw (40,35) node {$\times$};
            \draw (40,25) node {$\times$};
            \draw (30,25) node {$\times$};
            \draw (10,5) node {$\times$};
            \draw (55,80) rectangle (65,70);
            \draw (65,80) rectangle (75,70);
            \draw (55,70) rectangle (65,60);
            \draw (65,70) rectangle (75,60);
            \draw (45,60) rectangle (55,50);
            \draw (35,50) rectangle (45,40);
            \draw (25,40) rectangle (35,30);
            \draw (35,40) rectangle (45,30);
            \draw (25,30) rectangle (35,20);
            \draw (35,30) rectangle (45,20);
            \draw (45,70) rectangle (55,60);
            \draw (5,10) rectangle (15,0);
            \draw [dash pattern=on \pgflinewidth off 0.7mm](-5,90) rectangle (75,0);
    \end{tikzpicture}}
    \qquad
    \qquad
    \scalebox{0.7}{
        \begin{tikzpicture}[x=2\unitlength,y=2\unitlength,inner sep=0pt]   
            \draw (35,70) rectangle (45,60);
            \draw (35,60) rectangle (45,50);
            \draw (25,60) rectangle (35,50);
            \draw (25,70) rectangle (35,60);
            \draw (15,40) rectangle (25,30);
            \draw (15,30) rectangle (25,20);
            \draw (-5,30) rectangle (5,20);
            \draw (5,30) rectangle (15,20);
            \draw (-5,20) rectangle (5,10);
            \draw (5,20) rectangle (15,10);
            \draw (55,90) rectangle (65,80);
            \draw [dash pattern=on \pgflinewidth off 0.7mm](-5,90) rectangle (75,0);
            \draw (25,50) rectangle (35,40);
            \draw (60,85) node {$\triangle$};
            \draw (20,25) node {$\triangle$};
            \draw (20,35) node {$\triangle$};
            \draw (30,45) node {$\triangle$};
            \draw (30,55) node {$\triangle$};
            \draw (30,65) node {$\triangle$};
            \draw (40,65) node {$\triangle$};
            \draw (0,25) node {$\triangle$};
            \draw (0,15) node {$\triangle$};
            \draw (10,25) node {$\triangle$};
    \end{tikzpicture}}
    \caption{The diagrams $\Gamma'$ and $\Gamma''$}
    \label{fig:Gamma23}
\end{figure}
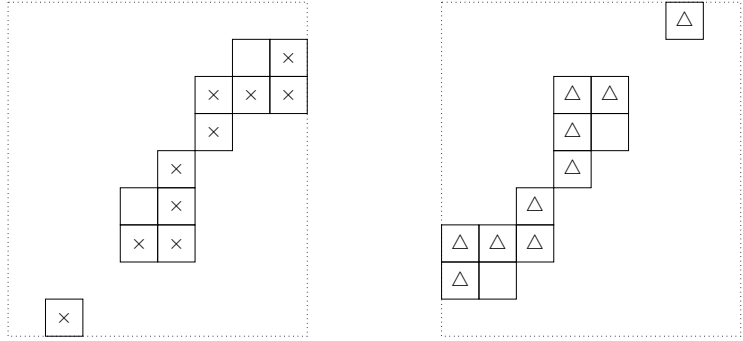

We can easily recover the dimensional result in \cite[Corollary 1.3]{arXiv:2411.07812} from \Cref{thm:dim}.

\begin{Corollary}
    If the Ferrers graph $\calG$ is complete bipartite, then $\dim(\KK[G(J_{\calG})])=2(a+b-2)$.
\end{Corollary}
\begin{proof}
    In this configuration, $\Gamma$ represents an $a\times b$ diagram. The cardinality $\#P$ can be explicitly expressed as $2(a+b-2)$.
\end{proof}

\begin{Question}
    An interesting problem is to classify when the binomial edge ring of a bipartite graph is a Koszul, Cohen--Macaulay, and normal ring.  
    Additionally, it would be valuable to investigate the configurations for which the binomial edge ring admits a \textsc{Sagbi} basis, such that its initial algebra has a defining ideal generated by a quadratic Gr\"obner basis with a squarefree initial ideal.  
    Furthermore, can one determine a dimension formula in terms of the combinatorial data of the bipartite graph?
\end{Question}

\begin{acknowledgment*}
    The first author was partially supported by the AMS-Simons Research Enhancement Grants for Primarily Undergraduate Institution Faculty 2024.  The second author acknowledges partial support from the ``the Fundamental Research Funds for Central Universities'' and ``the Innovation Program for Quantum Science and Technology'' (2021ZD0302902).
\end{acknowledgment*}

\noindent
\textbf{Data availability} Data sharing is not applicable to this article as no datasets were generated or analyzed
during the current study.

\bibliography{SkewDiagram}
\end{document}